\documentclass[12pt, a4paper]{amsart}
\usepackage{amsmath}
\usepackage{geometry,amsthm,graphics,tabularx,amssymb,shapepar}
\usepackage{amscd}
\usepackage{graphicx}
\usepackage{xypic}

\newcommand{\Fre}{{Fr\'{e}chet \,}} 

\newcommand{\BC}{{\mathbb {C}}}

\newcommand{\BN}{{\mathbb {N}}}

\newcommand{\BR}{{\mathbb {R}}}

\newcommand{\BZ}{{\mathbb {Z}}}

\newcommand{\CC}{{\mathcal {C}}}

\newcommand{\CS}{{\mathcal {S}}}
\newcommand{\CT}{{\mathcal {T}}}

\newcommand{\RM}{{\mathrm {M}}}

\newcommand{\Ad}{{\mathrm{Ad}}}

\newcommand{\GL}{{\mathrm{GL}}}

\newcommand{\Ind}{{\mathrm{Ind}}}
\newcommand{\ind}{{\mathrm{ind}}}

\newcommand{\Ker}{{\mathrm{Ker}}}

\newcommand{\Lie}{{\mathrm{Lie}}}

\newcommand{\rd}{{\mathrm{d}}}

\newcommand{\Sym}{{\mathrm{Sym}}}

\newcommand{\DR}{{\mathrm{DR}}}

\newcommand{\con}{\textit{C}}

\newcommand{\od}{\operatorname{d}}

\newcommand{\oH}{\operatorname{H}}

\newcommand{\oU}{\operatorname{U}}

\newcommand{\oD}{\operatorname{D}}

\newcommand{\g}{\mathfrak g}
\renewcommand{\k}{\mathfrak k}

\newcommand{\abs}[1]{\lvert#1\rvert}

\newcommand{\la}{\langle}
\newcommand{\ra}{\rangle}

\newcommand{\be}{\begin {equation}}
\newcommand{\ee}{\end {equation}}
\newcommand{\bee}{\begin {equation*}}
\newcommand{\eee}{\end {equation*}}

\theoremstyle{Theorem}

\theoremstyle{Theorem}
\newtheorem{lem}{Lemma}[section]
\newtheorem{corl}[lem]{Corollary}
\newtheorem{thml}[lem]{Theorem}

\newtheorem{prpl}[lem]{Proposition}

\theoremstyle{Theorem}
\newtheorem{prp}{Proposition}[section]
\newtheorem{corp}[prp]{Corollary}
\newtheorem{lemp}[prp]{Lemma}
\newtheorem{thmp}[prp]{Theorem}

\theoremstyle{Definition}
\newtheorem{dfn}{Definition}[section]
\newtheorem{prpd}[dfn]{Proposition}

\newtheorem{thmd}[dfn]{Theorem}
\newtheorem{dfnl}[lem]{Definition}

\theoremstyle{remark}
\newtheorem*{remark}{Remark}

\newtheorem*{remarks}{Remarks}

\theoremstyle{remark}

\newtheorem{examplel}[lem]{Example}

\begin{document}

\title[Schwartz homologies]{Schwartz homologies of representations of almost linear Nash groups }

\author[Y. Chen]{Yangyang Chen}
\address{School of Mathematics and Data Science, Jiangnan University, Wuxi, 214000, P. R. China}
\email{chenyy@jiangnan.edu.cn}

\
\author[B. Sun]{Binyong Sun}
\address{Institute for Advanced Study in Mathematics and New Cornerstone Science Laboratory, Zhejiang University, Hangzhou, 310058, P. R. China}
\email{sunbinyong@zju.edu.cn}


\maketitle

\begin{abstract}
Let $G$ be an almost linear Nash group, namely, a Nash group which admits a Nash homomorphism with finite kernel to some
$\GL_k(\mathbb R)$.  A homology theory (the Schwartz homology) is established for the category of smooth  \Fre representations of $G$ of moderate growth.  Frobenius reciprocity and Shapiro's lemma are proved  in this category.  As an application, we give a criterion for automatic extensions of Schwartz homologies of Schwartz sections of a tempered $G$-vector bundle. 
\end{abstract}

\setcounter{tocdepth}{2}
 \tableofcontents

\section{Introduction}
\subsection{Smooth representations}
Let us first recall the usual notions of representations and smooth representations of Lie groups.
Let $G$ be a Lie group.
By a representation of $G$, we mean a quasi-complete Hausdorff locally convex topological vector space $V$ over $\mathbb C$, together with a  continuous linear action
\be\label{actsm}
  G\times V\rightarrow V, \quad (g, v)\mapsto g.v.
\ee

The representation $V$ is said to be
 smooth if the map \eqref{actsm}  is smooth as a map of infinite dimensional manifolds.  The notion of smooth maps in infinite dimensional setting may be found in \cite{GN}, for example.
Note that the continuous linear action \eqref{actsm} is smooth if and only if the map
\be\label{actg}
 V\rightarrow V, \quad v\mapsto Y. v:=\lim_{t\rightarrow 0} \frac{ \exp(t Y).v-v}{t}
\ee
is defined and continuous for every $Y\in \Lie (G)$. When this is the case,  it is routine to check that \eqref{actg} defines a $\g$-module structure on $V$, which is called the differential of the representation. Here $\g:=\Lie (G)\otimes_\mathbb R \mathbb C$ denotes the complexified Lie algebra of $G$. Hence every smooth representation of $G$ is naturally a $\oU(\g)$-module. Here and as usual, $\oU$ indicates the universal enveloping algebra.

\begin{examplel}\label{exm12}

Let $M$ be a (finite dimensional, paracompact, Hausdorff) smooth manifold and let $E$ be a  quasi-complete Hausdorff locally convex topological vector space over $\mathbb C$. Then the space $C^{\infty}(M, E)$  of $E$-valued smooth functions on $M$ is a quasi-complete Hausdorff locally convex topological vector space over $\mathbb C$, under the usual smooth topology. The space
\[
 C^{\infty}_{c}(M, E)=\varinjlim_{\Omega\textrm{ is a compact subset of $M$}} C^{\infty}_\Omega(M, E)
\]
of the compactly supported smooth functions is also a quasi-complete Hausdorff locally convex topological vector space over $\mathbb C$, under the usual inductive topology. Here
\[
  C^{\infty}_\Omega(M, E):=\{f\in C^{\infty}(M, E)\,:\, \textrm{the support of $f$ is contained in $\Omega$}\}.
\]

Suppose that $M$ carries a smooth action of $G$ from left. Then for every smooth representation $E$ of $G$,  both $C^{\infty}(M, E)$ and $C^{\infty}_{c}(M, E)$ are smooth representations of $G$ under the action
\be\label{actgf}
  (g. f)(x):=g.(f(g^{-1}.x)), \quad g\in G, \, f\in C^{\infty}(M, E) \textrm{ or } C^{\infty}_{c}(M, E),\,x\in M.
\ee
Specifically, if $E$ is merely a quasi-complete Hausdorff locally convex topological vector space over $\mathbb C$, by viewing it as a smooth representation of $G$ with the trivial action, both
 $C^{\infty}(M, E)$ and $C^{\infty}_{c}(M, E)$ are smooth representations of $G$.

Likewise, if $M$ carries a smooth action of $G$ from right, then for every smooth representation $E$ of $G$,  both $C^{\infty}(M, E)$ and $C^{\infty}_{c}(M, E)$ are smooth representations of $G$ under the action
\be\label{actgf02}
  (g. f)(x):=g.(f(x.g)), \quad g\in G, \, f\in C^{\infty}(M, E) \textrm{ or } C^{\infty}_{c}(M, E),\,x\in M.
\ee

\end{examplel}

\subsection{Smooth cohomologies and smooth homologies}
We now review the basic theory of smooth cohomologies and smooth homologies, as respectively introduced in  \cite{HM} and \cite{BW}.

Denote by $\mathcal D\mathrm{mod}_G$ the category of smooth representations of $G$. The morphisms in this category are the $G$-intertwining  continuous linear maps. By using relatively injective resolutions in the category $\mathcal D\mathrm{mod}_G$,
 Hochschild and Mostow defined in  \cite{HM} a topological vector space $\oH^i(G; V)$ ($i\in \mathbb Z$) for every smooth representation $V$ of $G$, which was  called the smooth cohomology of $V$.
They  showed that the smooth cohomology agrees  with the usual continuous group cohomology \cite[Theorem 5.1]{HM}. If $G$ has only finitely many connected components, they also showed that the smooth cohomology agrees with the relative Lie algebra cohomology, namely, there is a topological identification (see \cite[Theorem 6.1]{HM})
\be\label{gkh0}
\oH^i(G; V)= \oH^i(\g, K; V),
\ee
where $K$ denotes a maximal compact subgroup of $G$. The reader is referred to \cite[(2.126)]{KV} and \cite[(2.127)]{KV} for the explicit complexes which  respectively compute the relative Lie algebra homology spaces and  the relative Lie algebra cohomology spaces.

By using strong projective resolutions in the category $\mathcal D\mathrm{mod}_G$,
Blanc and Wigner also defined in \cite{BW} a topological vector space $\oH_i(G; V)$, which was called the smooth homology of $V$.
The following Theorem plays a key role in the study of smooth homologies.

\begin{thml}
\label{thmA0}
Let $G$ be a Lie group. Let $E$ be a quasi-complete Hausdorff locally convex topological vector space over $\mathbb C$. Then
\[
\left\{f\in C^{\infty}_{c}(G, E):\int_G f(g)\rd_r g= 0\right\} = \sum_{g\in G}(g-1).C^{\infty}_{c}(G, E),
\]
where $G$ acts on $C^{\infty}_{c}(G, E)$ by the right translations.
Moreover,
\[
\left\{f\in C^{\infty}_{c}(G, E):\int_G f(g)\rd_r g= 0\right\} = \g.C^{\infty}_{c}(G, E) + \sum_{a\in A}(a-1).C^{\infty}_{c}(G, E),
\]
for every subset $A\subset G$ that intersects all connected components of $G$.
\end{thml}

Here and henceforth, $\rd_r g$ denotes a right invariant Haar measure on $G$. The Lie algebra $\g$ is identified with the space of left invariant complex vector fields on $G$ so that the action of $\g$ on $C^{\infty}_{c}(G, E)$  agrees with the differential of the $G$-action.
The first assertion of Theorem \ref{thmA0} is \cite[Theorem 1]{BW}, and the second assertion follows from the argument of \cite[Pages 264--266]{BW}.

\begin{remark}\label{remark1}
Many results in this article depend on the existence of integrals of vector-valued functions. More precisely, we will use  freely the following result in  \cite[Section 1, No.2, Corollary  of Proposition  5]{Bo}:
 Let $X$ be a locally compact Hausdorff topological space with a Borel measure $\mu$ on it. Let $E$ be a quasi-complete Hausdorff locally convex topological vector space over
$\mathbb C$. Then for every compactly supported continuous function $f: X\rightarrow E$, there is a unique  element
$\int_X f(x)\od\! \mu(x)\in E$ such that
\[
   \lambda\left(\int_X f(x)\od\! \mu(x)\right)=\int_X (\lambda\circ f)(x)\od\! \mu(x)
\]
for every continuous linear functional $\lambda: E\rightarrow \mathbb C$.

For example, in \cite{BW}, Blanc and Wigner proved Theorem \ref{thmA0}  with the assumption that $E$ is complete. Using the aforementioned result of the existence of the integrals of vector-valued functions,  their proof obviously extends to the case of quasi-complete spaces.
\end{remark}

Write
\be\label{delta}
\delta_G : G\rightarrow  \mathbb C^\times, \quad g\mapsto (\textrm{the absolute value of the determinant of $\Ad_g : \g\rightarrow \g$})
\ee
for the modular character of $G$. Here $\Ad$ indicates the adjoint representation of $G$ on $\g$.

If $G$ has only finitely many connected components, with a maximal compact subgroup $K$, Blanc and Wigner  showed the following  Poincar\'e duality theorem for smooth homologies and smooth cohomologies (see \cite[Theorem 3]{BW}):
\be\label{ohgv}
  \oH_i(G; V)=\oH^{n-i}(G; V\otimes \wedge^n (\g/\k)),
\ee
where $n:=\dim G/K$ and $\k$ is the complexified Lie algebra of $K$. Here the 1-dimensional space $\wedge^n (\g/\k)$ carries a representation of $G$ such that its restriction to $K$ is the adjoint representation, and its restriction to the identity connected component $G^\circ$ of $G$ corresponds to the modular character $\delta_{G^\circ}$. By \eqref{ohgv}, the study of smooth cohomologies is equivalent to the study of smooth homologies.
Recall that the relative Lie algebra homology and relative Lie algebra cohomology are related by the following
Poincar\'e duality (see \cite[Corollary 3.6]{KV}):
\[
\oH_i(\g, K; V)=\oH^{n-i}(\g, K; V\otimes \wedge^n (\g/\k)).
\]
Thus, by \eqref{gkh0}, the Poincar\'e duality \eqref{ohgv} is equivalent to
\be\label{gkh1}
\oH_i(G; V)= \oH_i(\g, K; V).
\ee


\subsection{Smooth \Fre representations of moderate growth}

For applications to the theory of automorphic forms, we are mostly interested in
smooth \Fre representations with certain growth conditions. In order to formulate the growth conditions precisely, it is  convenient to work in the setting of Nash manifolds and Nash groups.
The reader is referred  to  \cite{Sh} for the notions of Nash manifolds, Nash maps, Nash submanifolds, affine Nash manifolds, and the related notion of semialgebraic sets. Recall that a Nash group is a Nash manifold which is simultaneously a group such that the group multiplication map and the inversion map are both Nash maps. Nash groups are discussed in \cite{Sh2, Su, FS, Ca}, for examples.  A  group homomorphism between two Nash groups is called a Nash  homomorphism if it is also a Nash map.

Now suppose that $G$ is an almost linear Nash group, namely, a Nash group which admits a Nash homomorphism $G\rightarrow \GL_k(\mathbb R)$ with finite kernel, for some $k\geq 0$.
Structures of almost linear Nash groups were studied in detail in \cite{Su}.  A  representation  $V$ of $G$ is said to be of moderate growth, if for every continuous seminorm $|\cdot|_{\mu}$ on $V$, there is a positive Nash function $f$ on $G$ and a continuous seminorm $|\cdot|_{\nu}$ on $V$ such that
\[
|g.v|_{\mu}\leq f(g)|v|_{\nu}, \quad\text{for all}\quad g\in G, v\in V.
\]
It is said to be Fr\'{e}chet, if $V$ is \Fre as a topological vector space.
Denote by $\mathcal S\mathrm{mod}_G$ the category of smooth \Fre representations of $G$ of moderate growth. This is a full subcategory of
 $\mathcal D\mathrm{mod}_G$, and is the category of representations which we are mostly concerned with in this article.
 This category of representations was introduced and  studied in \cite{Fd} by F. du Cloux.

\begin{examplel}\label{exm123}

Let $M$ be a Nash manifold and let $E$ be a complex \Fre space. Then the  space of $E$-valued Schwartz functions on $M$, which is denoted by
$\CS(M, E)$, is naturally a  complex \Fre space. Moreover,
\[
   \CS(M, E)=\CS(M)\widehat \otimes E \qquad (\textrm{the completed projective tensor product}),
\]
where $\CS(M):=\CS(M, \mathbb C)$.
See Section \ref{secsf1} for details. If $M$ carries a left Nash action of $G$ and $E$ is a representation in $\CS\mathrm{mod}_G$, then $\CS(M, E)$ is  a representation in $\CS\mathrm{mod}_G$, under the action given as in \eqref{actgf}. Likewise, if  $M$ carries a right Nash action of $G$ and $E$ is a representation in $\CS\mathrm{mod}_G$, then $\CS(M, E)$ is also a representation in $\CS\mathrm{mod}_G$, under the action given as in \eqref{actgf02}.

\end{examplel}

Similar to Theorem \ref{thmA0}, the following theorem plays a key role in this article.
\begin{thml}
\label{thmA11}
Let $G$ be an  almost linear Nash group and let $E$ be a \Fre space. Then
\[
\left\{f\in \CS(G, E):\int_G f(g)\rd_r g= 0\right\} = \sum_{g\in G}(g-1).\CS (G, E),
\]
where $G$ acts on $\CS(G, E)$ by the right translations.
Moreover,
\[
\left\{f\in \CS(G, E):\int_G f(g)\rd_r g= 0\right\} =\g.\CS (G, E) + \sum_{a\in A}(a-1).\CS (G, E),
\]
for every subset $A\subset G$ that intersects all connected components of $G$.
\end{thml}

Here and as in Theorem \ref{thmA0}, $\g$ is identified with the space of left invariant complex vector fields on $G$.
We will give a proof of Theorem \ref{thmA11} in Section \ref{firs}.

\subsection{Schwartz homologies}

Recall that a homomorphism  $\alpha: V_1\rightarrow V_2$ of representations of $G$ is said to be
strong if there is a continuous linear map $\beta: V_2\rightarrow V_1$ such that $\alpha\circ \beta\circ \alpha=\alpha$  (see \cite{Ho} and \cite[Section 2]{HM}).

\begin{dfnl}\label{rp}
A representation $P$ in $\CS\mathrm{mod}_G$ is said to be relatively projective if for every surjective strong homomorphism $\alpha: V_1\rightarrow V_2$ and every homomorphism
$\beta: P\rightarrow V_2$ in $\CS\mathrm{mod}_G$, there exists a homomorphism $\tilde{\beta}:P\rightarrow V_1$ in $\CS\mathrm{mod}_G$ which lifts $\beta$, namely, $\alpha\circ\tilde{\beta}=\beta$.
\end{dfnl}

\begin{examplel}\label{exm1234}
Let the notations and assumptions  be as in Example \ref{exm123}. Suppose that $M$ is a principal left  $G$-Nash bundle, namely, $M$ carries a free  Nash action of $G$ from left with the following property: there is a Nash manifold $G\backslash M$ and a submersive Nash map $M\rightarrow G\backslash M$ whose fibers are the $G$-orbits in $M$. Then $\CS(M, E)$ is a relatively projective representation in $\CS\mathrm{mod}_G$. Likewise, if $M$ is a principal right  $G$-Nash bundle, then $\CS(M, E)$ is also a relatively projective representation in $\CS\mathrm{mod}_G$.
See Proposition \ref{repro}.
\end{examplel}

Write
\be\label{defodg}
\oD^\varsigma(G):=\CS(G) \od_r\! g
\ee
for the space of Schwartz densities on $G$. It is an associative algebra under convolutions. Put
\[
  \mathrm I^\varsigma(G):=\left\{ \mu \in \oD^\varsigma(G)\, :\, \int_G 1\od\! \mu(g)=0\right\}.
\]
This is a closed ideal of $\oD^\varsigma(G)$ of codimension 1. Every representation $V$ in $\CS\mathrm{mod}_G$ is a
$\oD^\varsigma(G)$-module under the action

\be\label{indact}
  \mu. v:= \int_G g. v\od\! \mu(g), \quad \mu\in \oD^\varsigma(G), \, v\in V.
\ee

Theorem \ref{thmA11} has the following consequence.

\begin{thml}
\label{thmA2}
Let $G$ be an  almost linear Nash group, and let $V$ be a representation in $\CS\mathrm{mod}_G$. Then
\[
  \mathrm I^\varsigma(G). V=\sum_{g\in G}(g-1).V=\g. V+\sum_{a\in A}(a-1).V,
\]
for every subset $A\subset G$ that intersects all connected components of $G$.
If $V$ is relatively projective in $\CS\mathrm{mod}_G$, then the above space is closed in $V$.
\end{thml}

Theorem \ref{thmA2} is important in our definition of Schwartz homology. We will prove it in Section \ref{co}.

In the notation of Theorem \ref{thmA2}, we write
\[
V_G:=V/ (\mathrm I^\varsigma(G). V)=V/\left(\sum_{g\in G}(g-1).V\right)=V/\left(\g. V+\sum_{a\in A}(a-1). V\right).
\]
By Theorem \ref{thmA2}, this is a \Fre space when $V$ is relatively projective in $\CS\mathrm{mod}_G$.
For a general representation $V$  in $\CS\mathrm{mod}_G$, we take a strong projective resolution
\be\label{prors}
\cdots \rightarrow P_2\rightarrow P_1\rightarrow P_0\rightarrow V\rightarrow 0
\ee
of $V$, namely, all $P_i$'s are relatively projective in $\CS\mathrm{mod}_G$,  all the arrows are strong homomorphisms, and  the above sequence is exact.
Define the $i$th ($i\in \mathbb Z$) Schwartz homology $\oH^\CS_i(G; V)$ of $V$ to be the $i$th homology of the complex
 \[
   \cdots \rightarrow (P_2)_G\rightarrow (P_1)_G\rightarrow (P_0)_G\rightarrow 0\rightarrow 0\rightarrow \cdots.
 \]
 Then $\oH^\CS_i(G; V)$ is a locally convex topological vector space which may or may not be Hausdorff. It is independent of the choice of the resolution \eqref{prors}. See Section \ref{sh} for details.
 For $i=0$,  there is a topological linear identification (see Proposition \ref{0homo})
 \[
  \oH^\CS_0(G; V)=V_G.
 \]

 In fact, Schwartz homologies agree with smooth homologies, as in the following theorem.

 \begin{thml}\label{hosho}
Let $G$ be an almost linear Nash group, and let $V$ be a representation in $\CS\mathrm{mod}_G$.  Then there is an identification
\[
\oH_{i}^\CS(G; V)=\oH_{i}(G; V)
\]
of topological vector spaces, for all $i\in \mathbb Z$.
\end{thml}

\begin{remark}
In view of the identification \eqref{gkh1}, Theorem \ref{hosho} is equivalent to saying that
\[
\oH_{i}^\CS(G; V)=\oH_{i}(\g,K; V).
\]
In particular, if $G$ is exponential in the sense that $G$ has no nontrivial compact subgroup, then $\oH_{i}^\CS(G; V)=\oH_{i}(\g; V)$  (the Lie algebra homology).
For a proof of Theorem \ref{hosho}, see Section \ref{iden}. In view of this theorem, all the results of this article can be thought of as results on smooth homologies (as defined by Blanc and Wigner).
\end{remark}

For applications to representation theory, it is important to show that  $\oH_{i}^\CS(G; V)$ is Hausdorff, at least  in some cases we are interested in. This is true when the homology space is finite dimensional, as claimed in the following proposition (see \cite[Proposition 6]{CW} and \cite[Lemma 3.4]{BoW}).

\begin{prpl}\label{intrhaus}
Let $G$ be an almost linear Nash group, and let $V$ be a representation in $\CS\mathrm{mod}_G$.  If $\oH_{i}^\CS(G; V)$ is finite dimensional ($i\in \BZ$), then it is Hausdorff.
\end{prpl}

\subsection{Frobenius reciprocity and Shapiro's lemma}

Let $H$ be a Nash subgroup of $G$, and let $V_0$ be a representation in $\CS\mathrm{mod}_H$. Let $H$ act  on
$\oD^\varsigma(G)$ by right translations. Then $ \oD^\varsigma(G)\widehat \otimes V_0$ with the diagonal $H$-action is a relatively projective representation in $\CS\mathrm{mod}_H$ (see Proposition \ref{repro}).
Define the Schwartz produced representation
\be\label{produced}
   \mathrm{pro}_H^G V_0:= ( \oD^\varsigma(G)\widehat \otimes V_0)_H,
\ee
which is a representation in $\CS\mathrm{mod}_G$. Here $G$ acts on  $\mathrm{pro}_H^G V_0$ through the left translations on $\oD^\varsigma(G)$. The Schwartz  produced representation \eqref{produced} is the same as the Schwartz indued representation as defined by du Cloux in \cite{Fd}, up to twists by characters (see Proposition \ref{pro=ind}). We refer the reader to Proposition \ref{geometry}, which explains the Schwartz induction in terms of bundles on the quotient space.

In many situations in representation theory of Lie groups, one is interested in Schwartz functions and Schwartz inductions instead of compactly supported smooth functions and compactly supported smooth inductions. For smooth homologies of compactly supported smooth inductions, Frobenius reciprocity and Shapiro's lemma were established in \cite[Theorem 11]{Bl}. However, in order to prove  Frobenius reciprocity and Shapiro's lemma for Schwartz produced representations, it is more natural to work in the setting of Schwartz homologies. This is the reason why we  introduce Schwartz homologies, although they agree with  smooth homologies by Theorem \ref{hosho}.

Precisely, for Schwartz  produced representations, we have the following version of Frobenius reciprocity.

\begin{thml}\label{fro}
Let $H$ be a Nash subgroup of an almost linear Nash group  $G$, and let $V_0$ be a representation in $\CS\mathrm{mod}_H$. Then the continuous linear map
\[
  \oD^\varsigma(G)\widehat \otimes V_0\rightarrow V_0, \quad \mu\otimes v\mapsto \int_G 1 \od\! \mu(g) \cdot v
\]
induces an identification
\[
  (\mathrm{pro}_H^G V_0)_G=(V_0)_H
\]
of topological vector spaces.
\end{thml}

As a corollary of Theorem \ref{fro}, we get the following version of Shapiro's lemma.

\begin{thml}\label{fro2}
Let $H$ be a Nash subgroup of an almost linear Nash group  $G$, and let $V_0$ be a representation in $\CS\mathrm{mod}_H$.  Then there is an identification
\[
\oH_{i}^\CS(G;  \mathrm{pro}_H^G V_0)=\oH_{i}^\CS(H; V_0)
\]
of topological vector spaces, for all $i\in \mathbb Z$.
\end{thml}

Theorems \ref{fro} and \ref{fro2} will be proved in Sections \ref{frobe} and \ref{shapi} respectively.

\subsection{Automatic extensions of Schwartz homologies}

Our original motivation to introduce Schwartz homologies was the applications to the calculations of invariant distributions. Precisely, let $M$ be a Nash manifold, and let $\mathsf E$ be  a  tempered vector bundle over $M$, as defined in  Section \ref{sectvb}. For example, all Nash vector bundles, as studied in \cite[Section 3.4]{AG1},
 are tempered vector bundles.  The \Fre space  $\Gamma^{\varsigma}(M,\mathsf E)$ of the Schwartz sections is defined in Section \ref{secsch}.

Now suppose that $M$ is a left $G$-Nash manifold, namely, it carries a  left Nash action $G\times M\rightarrow M$. Also suppose that $\mathsf E$ is a tempered left $G$-vector bundle, namely, it carries a  tempered bundle action $G\times \mathsf E\rightarrow \mathsf E$. Then
$\Gamma^{\varsigma}(M,\mathsf E)$ is naturally a representation of $G$ in $\CS\mathrm{mod}_G$. See Section \ref{sectvb00} for details.

Let $U$ be an open Nash submanifold of $M$.  Write $Z:=M\setminus U$. Extension by zero yields a closed  linear embedding (see Section \ref{secsch})
\be\label{extby}
\Gamma^{\varsigma}(U, \mathsf E|_U)\hookrightarrow\Gamma^{\varsigma}(M, \mathsf E),
\ee
and we identify
$\Gamma^{\varsigma}(U, \mathsf E|_U)$ with its image in $\Gamma^{\varsigma}(M, \mathsf E)$. Define
\[
\Gamma_{Z}^{\varsigma}(M, \mathsf E):=\Gamma^{\varsigma}(M, \mathsf E)/\Gamma^{\varsigma}(U, \mathsf E|_U).
\]

For every $z\in M$, let $G_z$ denote its stabilizer in $G$, and let $\mathsf E_z$ denote the fibre of $\mathsf E$ at $z$, which is a representation in $\CS\mathrm{mod}_{G_z}$.
Write
\[
  \mathrm{N}_{z}:=\frac{\mathrm T_z (M)}{\mathrm T_z(G.z)}\otimes_{\mathbb R} \mathbb C\qquad (\mathrm T_z\textrm{ stands for the tangent space})
\]
for the complexified  normal space, and write
\[
\mathrm{N}_{z}^*:=\textrm{the dual space of $\mathrm{N}_{z}$},
\]
which is the complexified  conormal space. They are both representations in $\CS\mathrm{mod}_{G_z}$.  Write
\[
\delta_{G/G_{z}}:=(\delta_G)|_{G_z} \cdot \delta_{G_z}^{-1} : G_z\rightarrow \mathbb C^\times.
\]
It is a positive Nash homomorphism.

Let $\chi: G\rightarrow \mathbb C^\times$ be a character which has moderate growth in the sense that $\abs{\chi}$ is bounded above by a positive Nash function on $G$. When no confusion is possible, we do not distinguish a 1-dimensional representation of a Lie group  with its corresponding character. In particular, $\chi$ is also viewed as a 1-dimensional representation in $\CS\mathrm{mod}_G$.

The following vanishing theorem plays an important role in the automatic extensions of Schwartz homologies and it also has some independent interest. For its proof, see Theorem \ref{thm333}.

 \begin{thml}\label{fro2.5}
Let the notation be as above.
Suppose that $U$ is $G$-stable and $Z$ has only finitely many $G$-orbits. Let $i\in \mathbb Z$.  Assume that for all $z\in Z$ and all integers $k\geq 0$,
\[
\left\{
  \begin{array}{l}
    \oH_{i+1}^\CS(G_{z}; \mathsf E_{z}\otimes\Sym^{k}( \mathrm{N}^*_{z})\otimes\delta_{G/G_{z}}\otimes\chi)\ \textrm{ is finite dimensional, and} \medskip \\
    \oH_i^\CS(G_{z}; \mathsf E_{z}\otimes\Sym^{k}( \mathrm{N}^*_{z})\otimes\delta_{G/G_{z}}\otimes\chi)=0,
  \end{array}
\right.
\]
 where $\Sym^{k}$ indicates the $k$th symmetric power. Then
\[
\oH^\CS_{i}(G; \Gamma_Z^{\varsigma}(M, \mathsf E)\otimes\chi)=0.
\]
\end{thml}

 Theorem \ref{fro2.5} has the following consequence of automatic extensions of Schwartz homologies.

\begin{thml}\label{fro3}
Let the notation and assumptions be as in Theorem \ref{fro2.5}. Then the  extension by zero homomorphism
\[
\Gamma^{\varsigma}(U,\mathsf E|_U)\hookrightarrow\Gamma^{\varsigma}(M,\mathsf E)
\]
induces an injective  continuous    linear map 
\be\label{im1}
 \oH^\CS_{i-1}(G; \Gamma^{\varsigma}(U,\mathsf E|_U)\otimes\chi)\rightarrow\oH^\CS_{i-1}(G; \Gamma^{\varsigma}(M,\mathsf E)\otimes\chi),
 \ee
and a surjective open continuous linear map
 \[
 \oH^\CS_{i}(G; \Gamma^{\varsigma}(U,\mathsf E|_U)\otimes\chi)\rightarrow\oH^\CS_{i}(G; \Gamma^{\varsigma}(M,\mathsf E)\otimes\chi).
 \]
 If moreover $ \oH_{i-1}^\CS(G_{z}; \mathsf E_{z}\otimes\Sym^{k}( \mathrm{N}^*_{z})\otimes\delta_{G/G_{z}}\otimes\chi)=0$ for all $z\in Z$ and all integers $k\geq 0$, then \eqref{im1} is a topological linear isomorphism. 
\end{thml}

The following lemma also plays an important role in the proof of Theorem \ref{fro3}.

\begin{lem}\label{open11}
Let $\phi : \CC_{\bullet}\rightarrow\CC^{\prime}_{\bullet}$ be a morphism of chain complexes of \Fre spaces. Let $i\in \BZ$. If the induced continuous linear map
\be\label{phi}
\phi: \oH_{i}(\CC_{\bullet})\rightarrow\oH_{i}(\CC^{\prime}_{\bullet})
\ee
 is surjective, then it must be an open map.
\end{lem}

To the best of our knowledge, Lemma \ref{open11} is new. See Lemma \ref{open} for a proof.

For every $z\in Z$, let
$N_z$ denote the unipotent radical of $G_z$ and $L_z$ denote a Levi factor of $G_z$ so that $G_z=L_z\ltimes N_z$. See Section \ref{secfinite} for more details.  Let $L_z^\circ$ denote the identity connected component of $L_z$.
Theorem \ref{fro3} has the following consequence for finite rank vector bundles.

\begin{thml}\label{fro4}
Let the notation be as above and let $i\in \BZ$. Assume  that
 \begin{itemize}
 \item $U$ is $G$-stable and $Z$ has only finitely many $G$-orbits;
   \item all the fibres of $\mathsf E$ are finite dimensional; and 
   \item for every $z\in  Z$ and every irreducible $L_z$-subrepresentation $F$ of
   \begin{equation} \label{homologys0}
   \bigoplus_{p,q\in \BZ, p+q=i} \oH^\CS_{p}(L_z^\circ;\BC)\otimes\oH^\CS_{q}(N_z;\BC),
   \end{equation}
  the  dual representation $F^\vee$, when viewed as a representation of $G_z$ via the trivial extension to $N_z$, does not occur as a subquotient in 
\be\label{sumsk}
\bigoplus_{k\geq 0} \mathsf E_{z}\otimes\Sym^{k}( \mathrm{N}^*_{z})\otimes\delta_{G/G_{z}}\otimes\chi.
\ee
\end{itemize}
Then the  extension by zero homomorphism
\[
\Gamma^{\varsigma}(U,\mathsf E|_U)\hookrightarrow\Gamma^{\varsigma}(M,\mathsf E)
\]
induces an injective  continuous    linear map
\[
 \oH^\CS_{i-1}(G; \Gamma^{\varsigma}(U,\mathsf E|_U)\otimes\chi)\rightarrow\oH^\CS_{i-1}(G; \Gamma^{\varsigma}(M,\mathsf E)\otimes\chi),
 \]
and a surjective open continuous linear map
 \be\label{im2}
 \oH^\CS_{i}(G; \Gamma^{\varsigma}(U,\mathsf E|_U)\otimes\chi)\rightarrow\oH^\CS_{i}(G; \Gamma^{\varsigma}(M,\mathsf E)\otimes\chi).
 \ee
 If moreover for every $z\in  Z$ and every irreducible $L_z$-subrepresentation $E$ of
\[
 \bigoplus_{p,q\in \BZ, p+q=i+1} \oH^\CS_{p}(L_z^\circ;\BC)\otimes\oH^\CS_{q}(N_z;\BC),
 \]
   the  dual representation $E^\vee$, when viewed as a representation of $G_z$ via the trivial extension to $N_z$, does not occur as a subquotient in \eqref{sumsk},  then \eqref{im2} is a topological linear isomorphism.
\end{thml}

 Note that \eqref{homologys0} is naturally a completely reducible representation of $L_z$.

We will prove Theorems \ref{fro3} and \ref{fro4} in the last section of this article.
In the $p$-adic case, a result similar to Theorem \ref{fro4} was established in \cite[Theorem 1.4]{HS}.
We give two  examples to illustrate the usefulness of Theorem \ref{fro4}.
The first one comes from Tate's thesis.

\begin{examplel}\label{tate}
The \Fre space $\CS(\BR)$ is a representation in $\CS\mathrm{mod}_{\BR^\times}$ with the action
\[
  (t.f)(x)=f(xt), \quad t\in \BR^\times, \, f\in \CS(\BR), \, x\in \BR.
\]
Suppose $\chi$ is a character of $\BR^\times$ which does not have the form
\be\label{chialg}
  t\mapsto t^{-k}, \quad k\geq 0.
\ee
Then by applying  Theorem \ref{fro4} to the trivial bundle $\BR\times \BC$ over $\BR$, we know that the obvious embedding
\[
  \CS(\BR^\times)\hookrightarrow \CS(\BR)
\]
induces a topological linear isomorphism
\[
 \oH^\CS_{i}(\BR^\times; \CS(\BR^\times)\otimes\chi)\rightarrow\oH^\CS_{i}(\BR^\times; \CS(\BR)\otimes\chi), \quad \textrm{for all $i\in \mathbb Z$}.
\]
Using Theorem \ref{fro2}, we get a topological linear isomorphism
\be\label{htate}
  \oH^\CS_{i}(\BR^\times; \CS(\BR)\otimes\chi)\cong\left\{
                     \begin{array}{ll}
                        \BC, & \textrm{if $i=0$};\\
                        \{0\}, &  \textrm{if $i\neq 0$}.
                        \end{array}
                        \right.
\ee

\end{examplel}

\begin{remark}
If fact, \eqref{htate} holds even when $\chi$ has the form \eqref{chialg}. We leave the proof to the interested reader.
\end{remark}

The second example is about Whittaker models.

\begin{examplel}\label{whittaker}
Suppose that $G$ is the real points of a quasi-split connected reductive linear algebraic group defined over $\BR$. Let $B$ be a Borel subgroup of $G$ whose unipotent radical is denoted by $N$.
Let $\psi: N\rightarrow \BC^\times$ be a unitary character which is non-degenerate in the sense that
\[
  \psi|_{N\cap B'}\neq 1\quad \textrm{for every Borel subgroup $B'$ of $G$ which is not opposite to $B$}.
\]
Let $J$ be a smooth principal series representation of $G$. We claim that there is  a topological linear isomorphism
\be\label{htate2}
  \oH^\CS_{i}(N; J\otimes\psi)\cong\left\{
                     \begin{array}{ll}
                        \BC, & \textrm{if $i=0$};\\
                        \{0\}, &  \textrm{if $i\neq 0$}.
                        \end{array}
                        \right.
\ee
This particularly implies the uniqueness of the Whittaker models. See also  \cite[Theorems 6.2 and 9.1]{CHM}.

In fact, suppose that $M$ is the set of Borel subgroups of $G$, which is naturally a left $G$-Nash manifold. Then $J=\Gamma^{\varsigma}(M,\mathsf E)$, for a certain  tempered left $G$-vector bundle $\mathsf E$ of rank one over $M$. Suppose that $U$ is the open $N$-orbit in $M$. Fix a base point of $U$ and an $N$-equivariant  trivialization of $\mathsf E|_U$. Then
\[
    \Gamma^{\varsigma}(U,\mathsf E)=\CS(N).
\]
Now \eqref{htate2} follows from Theorems \ref{fro4} and \ref{fro2}.

This example shows that, at least for the study of Whittaker models, it is more natural to use Schwartz inductions and Schwartz homologies instead of compactly supported smooth inductions and smooth homologies.
\end{examplel}

\subsection{Structure of this article}
We will introduce some preliminaries on several function spaces in Section \ref{pfunction}. These include Schwartz functions on Nash manifolds with values in \Fre spaces, linear families of moderate growth, tempered linear families and tempered vector bundles over Nash manifolds. In Section \ref{prep}, we show that the action map of every representation in the category $\CS\mathrm{mod}_G$ gives a tempered linear family, and discuss homogeneous tempered vector bundles. Then we will prove the first main result of this article in Section \ref{firs}. In Section \ref{sh}, we introduce the notions of relatively projective representations, strong projective resolutions and Schwartz homologies of representations of an almost linear Nash group. Along the way, we will give a proof of Theorem \ref{thmA2} in Section \ref{co}.

In Section \ref{sectvb00}, we introduce Schwartz sections of tempered vector bundles. We will also recall the Schwartz induced representations introduced by du Cloux and show that it is isomorphic to Schwartz produced representations as defined in \eqref{produced}, up to twists by characters. In Section \ref{frobe}, we prove the Frobenius reciprocity law, namely, Theorem \ref{fro}.
In Section \ref{secsha}, we establish properties of the Schwartz induction functor and then prove Shapiro's lemma, namely, Theorem \ref{fro2}. Moreover, we show that the Schwartz homologies coincide with the relative Lie algebra homologies (Theorem \ref{BBB}). As an application of these results, we prove the automatic extensions of Schwartz homologies, namely, Theorems \ref{fro3} and \ref{fro4}, in the last section.

{\bf Acknowledgements}
The authors would like to thank the anonymous referee for the valuable suggestions to improve the manuscript. They are grateful to Professor Jun Yu for pointing out the mistakes in Theorem 1.15 and Proposition 7.11 in  the published version.
B. Sun was supported in part by National Natural Science Foundation of China grants 11525105, 11688101, 11621061, and 11531008. 

\section{Preliminaries on some function spaces}\label{pfunction}

\subsection{Schwartz functions} \label{secsf1}

Let $M$ be a Nash manifold. Recall that a differential operator $D$ on $M$  is said to be Nash if $D(f)$ is a Nash function on $U$, for every (complex valued) Nash function $f$ on every  open Nash submanifold $U$ of $M$. See \cite[Section 3.5]{AG1} for more details.

Recall that a Nash manifold is said to be  affine if  it is Nash isomorphic to some closed Nash submanifolds of $\mathbb R^k$ for some $k\geq 0$.  It is known that every open Nash submanifold of every affine Nash manifold  is also affine. See \cite[Proposition III.1.7]{Sh}
and \cite[Section 2.22]{Sh2} for details.

Let $E$ be a complex \Fre space. If $M$ is affine, set
\begin{eqnarray*}
  \CS(M, E)&:=& \{ f\in C^{\infty}(M, E)\, : \, \sup_{x\in M}\abs{(Df)(x)}_\nu <\infty \textrm{ for all Nash differential  }  \\
  & &\left. \  \textrm{ operators $D$ on $M$ and all continuous seminorms $\abs{\,\cdot\,}_\nu$ on $E$ }\right \}.
\end{eqnarray*}
Then  $\CS(M, E)$ is a complex \Fre space with the obvious topology. In general, take a finite covering $\{M_i\}_{i=1}^k$ ($k\geq 0$) of $M$ by affine open Nash submanifolds. Then by extension by zero, we get a continuous linear map
\[
  \bigoplus_{i=1}^k \CS(M_i, E)\rightarrow C^{\infty}(M, E).
\]
We define $\CS(M, E)$ to be the image of this map, equipped with the quotient topology of the domain. Then $\CS(M, E)$ is a \Fre space which is independent of the covering $\{M_i\}_{i=1}^k$ (see \cite[Proposition 5.1.2]{AG1}). This is called the space of $E$-valued Schwartz functions on $M$.

For simplicity, we write $\CS(M):=\CS(M, \BC)$. This is a nuclear \Fre space. An easy argument of functional analysis shows that (see  \cite[Proposition 1.2.6]{Fd})
\[
  \CS(M, E)=\CS(M)\widehat \otimes E\qquad(\textrm{the completed projective tensor product}).
\]
We refer the reader to \cite[Section 3]{Ta} for more details about topological tensor products.

\subsection{Linear families of moderate growth}

When $M$ is affine, a  function $f: M\rightarrow \BC$ is said to be of  moderate growth if $ \abs{f}$ is bounded above by a positive Nash function on $M$. Here $f$ may or may not be continuous.

\begin{lem}\label{mgr}
Suppose that the Nash manifold $M$ is affine. Let $\{M_i\}_{i=1}^k $ ($k\geq 0$) be a finite covering of $M$ by open Nash submanifolds.
Then a function $f: M\rightarrow \BC$ is of moderate growth if and only if $f|_{M_i}$ is so for every $1\leq i\leq k$.
\end{lem}

\begin{proof}
See \cite[Theorems 4.5.1 and 4.5.2]{AG1}.
\end{proof}

Let $E_1$ and $E_2$ be two \Fre spaces. A map $\phi: M\times E_1\rightarrow E_2$ is called a liner family if the map $\phi(x, \,\cdot\,): E_1\rightarrow E_2$ is linear for all $x\in M$.
Generalizing the previous notion of moderate growth, we introduce the following definition.

\begin{dfnl}\label{defmgr}
Suppose that the Nash manifold $M$ is affine. A linear family $\phi: M\times E_1\rightarrow E_2$ is said to be of  moderate growth if for
every continuous seminorm $\abs{\,\cdot\,}_2$ on $E_2$, there is a positive Nash function $f$ on $M$ and  a continuous seminorm $\abs{\,\cdot\,}_1$ on $E_1$ such that
\[
  \abs{\phi(x,u)}_2\leq f(x) \abs{u}_1\quad \textrm{for all } x\in M, u\in E_1.
\]

\end{dfnl}

\begin{lem}\label{mgr2}
Suppose that the Nash manifold $M$ is affine. Let $\{M_i\}_{i=1}^k $ ($k\geq 0$) be a finite covering of $M$ by open Nash submanifolds.
Then a linear family  $\phi: M\times E_1\rightarrow E_2$ is of moderate growth if and only if $\phi|_{M_i\times E_1}$ is so for every $1\leq i\leq k$.
\end{lem}
\begin{proof}
The only if part of the lemma is obvious. To prove the if part, assume that  $\phi|_{M_i\times E_1}$ is  of moderate growth for every $1\leq i\leq k$.
Let $\abs{\,\cdot\,}_2$ be a continuous seminorm on $E_2$.
Then there is a positive Nash function $f_i$ on $M_i$ and  a continuous seminorm $\abs{\,\cdot\,}_{1,i}$ on $E_1$ such that
\[
  \abs{\phi(x,u)}_2\leq f_i(x) \abs{u}_{1,i}\quad \textrm{for all } x\in M_i, u\in E_1.
\]
For each $x\in M$, define
\[
  f(x):=\min_{1\leq j\leq k, x\in M_j} f_j(x).
\]
Then Lemma \ref{mgr} implies that the function $f$ on $M$ is of moderate growth. Define a continuous seminorm $\abs{\, \cdot\,}_1$ on $E_1$ by
\[
\abs{u}_{1}:=\sum_{i=1}^k \abs{u}_{1,i}, \quad u\in E_1.
\]
Then
\[
  \abs{\phi(x,u)}_2\leq f(x) \abs{u}_{1}\quad \textrm{for all } x\in M, u\in E_1.
\]
This proves the lemma.
\end{proof}
In general when $M$ may or may not be affine, we make the following definition.

\begin{dfnl}\label{defmgr2}
A linear family $\phi: M\times E_1\rightarrow E_2$ is said to be of  moderate growth if there is a finite covering  $\{M_i\}_{i=1}^k $ ($k\geq 0$) of $M$ by affine open Nash submanifolds such that
 $\phi|_{M_i\times E_1}$ is of moderate growth for all $1\leq i\leq k$.
\end{dfnl}

By Lemma \ref{mgr2}, Definition \ref{defmgr2} agrees with Definition \ref{defmgr} when $M$ is affine. Moreover, Lemma \ref{mgr2} remains true if we allow $M$ to be affine or not.

Let $E_3$ be another \Fre space. The following lemma is easy to check.

\begin{lem}\label{mgrcom}
Let $\phi_1: M\times E_1\rightarrow E_2$  and $\phi_2: M\times E_2\rightarrow E_3$  be linear families of moderate growth. Then the linear family
\[
   \phi_3 : M\times E_1\rightarrow E_3, \quad (x,u)\mapsto \phi_2(x, \phi_1(x, u))
\]
is of moderate growth.
\end{lem}

\subsection{Tempered linear families}

We introduce the following definition.

\begin{dfnl}\label{deftem}
Suppose that $M$ is affine. A linear family $\phi: M\times E_1\rightarrow E_2$ is said to be tempered if
\begin{itemize}
\item
it is smooth as a map of infinite dimensional manifolds; and
\item
for every Nash differential operator $D$ on $M$, the linear family
\[
 D\phi: M\times E_1\rightarrow E_2
\]
 is of moderate growth.
 \end{itemize}

\end{dfnl}

The following lemma is an analogue of Lemma \ref{mgr2}.

\begin{lem}\label{tem2}
Suppose the Nash manifold $M$ is affine. Let $\{M_i\}_{i=1}^k $ ($k\geq 0$) be a finite covering of $M$ by open Nash submanifolds.
Then a linear family  $\phi: M\times E_1\rightarrow E_2$ is tempered if and only if $\phi|_{M_i\times E_1}$ is so for every $1\leq i\leq k$.
\end{lem}

\begin{proof}
The if part of the lemma follows directly from Lemma \ref{mgr2}. The only if part is proved as in the proof of
\cite[Theorem 4.5.1]{AG1}.
\end{proof}

Similar to Definition \ref{defmgr2}, we make the following definition when $M$ may or may not be affine.

\begin{dfnl}\label{deftem2}
A linear family $\phi: M\times E_1\rightarrow E_2$ is said to be tempered if there is a finite covering  $\{M_i\}_{i=1}^k $ ($k\geq 0$) of $M$ by affine open Nash submanifolds such that
 $\phi|_{M_i\times E_1}$ is tempered for all $1\leq i\leq k$.
\end{dfnl}

By Lemma \ref{tem2}, Definition \ref{deftem2} agrees with Definition \ref{deftem} when $M$ is affine. Moreover, Lemma \ref{tem2} remains true if we allow $M$ to be affine or not.

Similar to Lemma \ref{mgrcom}, we have the following lemma.

\begin{lem}\label{temcom}
Let $\phi_1: M\times E_1\rightarrow E_2$  and $\phi_2: M\times E_2\rightarrow E_3$ be tempered linear families. Then the linear family
\[
 \phi_3 :   M\times E_1\rightarrow E_3, \quad (x,u)\mapsto \phi_2(x, \phi_1(x, u))
\]
is also tempered.
\end{lem}

\begin{proof}
In view of Lemma \ref{tem2}, we assume without loss of generality that $M$ is an open Nash submanifold of $\BR^n$ ($n\geq 0$). Let $D$ be a Nash differential operator on $M$.
Then by Leibniz rule, there are Nash differential operators $D_1, D_2, \cdots, D_k$ and $D'_1, D'_2, \cdots, D'_k$ ($k\geq 0$) such that
\[
  (D\phi_3)(x, u)=\sum_{i=1}^k (D'_i\phi_2)(x, (D_i\phi_1)(x,u)),\quad \textrm{for all } x\in M, \, u\in E_1.
\]
Hence the lemma follows from Lemma \ref{mgrcom}.
\end{proof}

The following lemma generalizes the fact that the pullback of a tempered function through a Nash map is also tempered.

\begin{lem}\label{temcompullb}
Let $\phi_1: M_1\times E_1\rightarrow E_2$ be a tempered linear family, where $M_1$ is a Nash manifold. Then for every Nash map $f: M\rightarrow M_1$,  the linear family
\[
 \phi:   M\times E_1\rightarrow E_2, \quad (x,u)\mapsto \phi_1(f(x), u)
\]
is also tempered.
\end{lem}

\begin{proof}
In view of Lemma \ref{tem2}, we assume without loss of generality that $M$ is an open Nash submanifold of $\BR^n$ ($n\geq 0$), and $M_1$ is an open Nash submanifold of $\BR^m$ ($m\geq 0$).

Write
\[
  f(x)=(f_1(x), f_2(x),\cdots, f_m(x)), \quad x\in M.
\]
For each  $I=(i_1, i_2,\cdots, i_n)\in \BN^n$ ($\BN$ denotes the set of nonnegative integers), write $\partial_I=\partial_1^{i_1} \partial_2^{i_2}\cdots \partial_n^{i_n}$, to be viewed as a differential operator on $M$.
By the chain rule, for each $1\leq i\leq n$,
\be\label{chainrule}
 ( \partial_i \phi)(x,u)=\sum_{k=1}^m  (\partial_i f_k)(x)\cdot  (\partial_k\phi_1)(f(x),u), \quad x\in M, \, u\in E_1.
\ee
By using the Leibniz rule and \eqref{chainrule} inductively, we know that the function $(\partial_I\phi)(x,u)$ is a finite sum of functions of the form
\[
  (\partial_{I_1} f_{k_1})(x)\cdot  (\partial_{I_2} f_{k_2})(x)\cdot \ldots \cdot  (\partial_{I_r} f_{k_r})(x)\cdot   (\partial_{J}\phi_1)(f(x),u),
\]
where $r\geq 0$, $I_1, I_2, \cdots, I_r, J\in \BN^n$, $1\leq k_1, k_2, \cdots, k_r\leq m$. Thus the linear family $\partial_I\phi: M\times E_1\rightarrow E_2$ is of moderate growth, and the Lemma follows.
\end{proof}

\begin{dfnl}\label{tembdm}
Let $M_1$, $M_2$ be Nash manifolds. A map $M_1\times E_1\rightarrow M_2\times E_2$ is called a tempered bundle map if it has the form
\[
  (x, u)\mapsto (f(x), \phi(x, u)),
\]
where $f: M_1\rightarrow M_2$ is a Nash map, and $\phi: M_1\times E_1\rightarrow E_2$ is a tempered linear family.
\end{dfnl}

\begin{lem}\label{temcombu}
Let $\psi_1: M_1\times E_1\rightarrow M_2\times E_2$  and  $\psi_2: M_2\times E_2\rightarrow M_3\times E_3$ be tempered bundle maps, where
$M_1$, $M_2$, $M_3$ are Nash manifolds, and $E_1$, $E_2$, $E_3$ are \Fre spaces. Then
\[
\psi_2\circ \psi_1: M_1\times E_1\rightarrow M_3\times E_3
\]
is also a tempered bundle map.
\end{lem}

\begin{proof}
Write $\psi_1 = (f_1,\phi_1)$ and $\psi_2 = (f_2,\phi_2)$, with notations as in Definition \ref{tembdm}.
Define a linear family
\[
  \phi'_2: M_1\times E_2\rightarrow E_3, \quad (x,v)\mapsto \phi_2(f_1(x),v).
\]
By Lemma \ref{temcompullb}, this linear family is tempered. Note that
\[
  (\psi_2\circ \psi_1)(x, u)=((f_2\circ f_1)(x), \phi_2'(x, \phi_1(x, u))).
\]
Hence the lemma follows by Lemma \ref{temcom}.
\end{proof}

\subsection{Tempered vector bundles}\label{sectvb}

Let $M$ be a Nash manifold and let $\mathsf E$ be a \Fre bundle over $M$, namely, a topological vector bundle over $M$ such that all the fibres are  \Fre spaces.
A local chart of $\mathsf E$ is defined to be a triple $(U,E, \phi)$, where $U$ is an open Nash submanifold of $M$, $E$ is a fibre of
$\mathsf E$, and
\[
  \phi: U\times E\rightarrow \mathsf E|_U
 \]
is a topological isomorphism of vector bundles over $U$, where $\mathsf E|_U$ denotes the restriction of $\mathsf E$ to $U$, which is a  topological vector bundle over $U$.

\begin{dfnl}\label{deftemst}
A tempered  structure on $\mathsf E$ is a subset $\CT_\mathsf E$ of the set of all local charts of $\mathsf E$ with the following properties:
\begin{itemize}
\item every two elements  $(U_1,E_1, \phi_1)$, $(U_2, E_2, \phi_2)$ in $\CT_\mathsf E$ are compatible in the sense that
the map
\[
 \phi^{-1}_2\circ \phi_1:  (U_1\cap U_2)\times E_1\rightarrow (U_1\cap U_2)\times E_2
\]
and its inverse are both tempered bundle maps;
\item  for every local chart of $\mathsf E$, if it is compatible with all elements of $\CT_\mathsf E$, then it belongs to $\CT_\mathsf E$;
\item  there exists a finite family $\{(U_i, E_i, \phi_i)\}_{i=1}^k$ ($k\geq 0$) of elements of $\CT_\mathsf E$ such that
     $\{U_i\}_{i=1}^k$ is a covering of $M$.
\end{itemize}
\end{dfnl}

\begin{remark}
Recall that the notion of tempered bundle maps between trivial \Fre bundles over Nash manifolds has been defined in Definition \ref{tembdm}.
Suppose that there is a  finite family $\{(U_i, E_i, \phi_i)\}_{i=1}^k$ ($k\geq 0$) of pairwise compatible local charts  of $\mathsf E$ such that
     $\{U_i\}_{i=1}^k$ is a covering of $M$. Then by Lemmas \ref{tem2} and \ref{temcombu}, all the local charts of $\mathsf E$ that are compatible with all $(U_i, E_i, \phi_i)$'s  form a tempered structure on $\mathsf E$.
\end{remark}

\begin{dfnl}\label{deftembundle}
A tempered vector bundle is a triple $(M, \mathsf E, \CT_\mathsf E)$, where $M$ is a Nash manifold, $\mathsf E$ is a \Fre bundle over $M$ and $\CT_\mathsf E$ is a tempered structure on $\mathsf E$.
\end{dfnl}

When $\CT_\mathsf E$ is understood,   we call $\mathsf E$ a tempered vector bundle over $M$.
Obviously every trivial \Fre bundle over a Nash manifold is canonically a tempered vector bundle. For every tempered vector bundle $(M, \mathsf E, \CT_\mathsf E)$ and  every Nash submanifold $Z$ of $M$,  $\mathsf E|_Z$ is obviously a tempered vector bundle over $Z$.

Generalizing the notion of tempered bundle maps between trivial \Fre bundles as given in Definition \ref{tembdm}, we make the following definition.

\begin{dfnl}\label{temmor}
Let $(M_1, \mathsf E_1, \CT_{\mathsf E_1})$ and $(M_2, \mathsf E_2, \CT_{\mathsf E_2})$ be two tempered vector bundles.
A map $f: \mathsf E_1\rightarrow \mathsf E_2$ is called a tempered bundle map if there is a Nash map $f_0: M_1\rightarrow M_2$ such that
the diagram
\[
 \begin{CD}
           \mathsf E_1@>f >> \mathsf E_2 \\
            @V VV           @VV V\\
           M_1@>f_0>> M_2\\
  \end{CD}
\]
commutes, and for every $(U_1, E_1, \phi_1)\in \CT_{\mathsf E_1}$ and $(U_2, E_2, \phi_2)\in \CT_{\mathsf E_2}$ with $f_0(U_1)\subset U_2$, the map
\[
\phi_2^{-1} \circ  f\circ \phi_1: U_1\times E_1\rightarrow U_2\times E_2
\]
is a tempered bundle map in the sense of Definition \ref{tembdm}.

\end{dfnl}

By Lemmas \ref{tem2} and \ref{temcombu}, Definition \ref{temmor} agrees with Definition \ref{tembdm} for trivial \Fre bundles.

\begin{prpl}
The composition of two tempered bundle maps between tempered vector bundles is also a tempered bundle map.
\end{prpl}

\begin{proof}
 This follows easily from Lemma \ref{temcombu}.
\end{proof}


\section{Representations and homogeneous tempered vector bundles}\label{prep}

\subsection{The action of compactly supported distributions} \label{secsf2}

Let $G$ be a Lie group. Recall from the Introduction that $\g$ denotes the complexified Lie algebra of $G$.
Let $\RM_c(G)$ denote the space of compactly supported Borel measures on $G$. It is an associative algebra under convolutions. Every representation $V$ of $G$ is naturally an $\RM_c(G)$-module:
 \be\label{intm}
  \mu.v:=\int_G g.v\, \rd \mu(g), \qquad \mu\in \RM_c(G), \, v\in V.
\ee

 Let $\oD_c^{-\infty}(G)$ denote the space of compactly supported distributions on $G$. It is an associative algebra under convolutions, and contains both $\RM_c(G)$ and $\oU(\g)$ as subalgebras. Moreover, by the structure theory of compactly supported distributions, we have that
 \be\label{dgmg}
   \oD_c^{-\infty}(G)=\RM_c(G)\cdot \oU(\g)=\oU(\g)\cdot \RM_c(G).
 \ee
 Every smooth representation $V$ of $G$ is naturally a $\oD_c^{-\infty}(G)$-module by requiring that
 \be\label{defetav}
  \lambda( \eta. v)= \la \eta, \phi_{\lambda, v}\ra, \qquad \eta\in \oD_c^{-\infty}(G), \, v\in V,
 \ee
 for all continuous linear functionals $\lambda$ on $V$. Here $\phi_{\lambda, v}\in \con^{\,\infty}(G)$ denotes the matrix coefficient $g\mapsto \lambda(g.v)$.
 This action of $\oD_c^{-\infty}(G)$ extends the existing actions of $G$, $\RM_c(G)$ and $\oU(\g)$ on $V$.

\begin{remark}
The existence of $\eta.v\in V$ satisfying \eqref{defetav} follows from \eqref{dgmg}, and the previously defined  actions of $\RM_c(G)$ and $\oU(\g)$  on $V$.
\end{remark}

\subsection{The category $\mathcal{S}\mathrm{mod}_{G}$}

In the rest of this article, suppose that $G$ is an almost linear Nash group as in the Introduction.

\begin{lem}\label{smodg}
Let
\be\label{actge}
  G\times E\rightarrow E, \quad (g,u)\mapsto g.u
\ee
be a  linear action of $G$ on a \Fre space $E$. If the  map \eqref{actge} is smooth, and has moderate growth as a linear family, then it is tempered as a linear family.
\end{lem}

\begin{proof}
This is known to experts. The proof follows by using the identity
\[
  (X. \phi)(g, u)=\phi(g, X.u), \quad \textrm{for all } X\in \oU(\g), \, g\in G, \, u\in E,
\]
where $\phi$ denotes the map \eqref{actge}.
\end{proof}

As in the Introduction, let $\mathcal{S}\mathrm{mod}_{G}$  denote the category of smooth \Fre representations of $G$ of moderate growth. By Lemma \ref{smodg}, the action map of every representation in $\mathcal{S}\mathrm{mod}_{G}$ is a tempered linear family.

The following lemma is easily checked.

\begin{lem}\label{sq}
Let $V$ be a representation in $\mathcal{S}\mathrm{mod}_{G}$. Then all subrepresentations and  quotient representations of $V$ are representations in $\mathcal{S}\mathrm{mod}_{G}$.
\end{lem}

\subsection{Homogeneous tempered vector bundles}\label{htvb}

Let $M$ be a left $G$-Nash manifold, namely, a Nash manifold carrying a  left Nash action $G\times M\rightarrow M$. By a tempered left $G$-vector bundle over $M$, we mean a tempered vector bundle
$\mathsf E$ over $M$, together with an action $G\times \mathsf E\rightarrow \mathsf E$ which is a tempered bundle map. Here $G\times \mathsf E$ is obviously viewed as a tempered vector bundle over $G\times M$.

Given a Nash subgroup $H$ of $G$,  and  a representation $V_0$ of $H$ in the category $\CS\mathrm{mod}_{H}$, in what follows we define a canonical tempered structure $\CT_{G\times^{H}V_0 }$ on the topological vector bundle $G\times ^{H}V_0$ over the Nash manifold
$G/H$. Here and as usual,  $G\times ^{H}V_0$ denotes the orbit space of the action
\[
  H\times (G\times V_0)\rightarrow G\times V_0, \quad (h, (g,v))\mapsto (gh^{-1}, h.v).
\]

Write $\pi : G\rightarrow G/H$ for  the quotient map.  It is a surjective submersive Nash map. By \cite[Theorem 2.4.3]{AG3}, there exists a finite open cover
\[
G/H = \cup_{i=1}^k U_i \quad (k\geq 0)
\]
 by open Nash submanifolds of $G/H$ such that $\pi$ has a Nash section $s_i$ on each $U_i$. It is easy to check that
the map
\[
\begin{array}{rcl}
\psi_i:U_i\times V_0&\rightarrow& (G\times^{H} V_0)|_{U_i}\\
(x,v)&\mapsto& (s_i(x),v)
\end{array}
\]
is a topological isomorphism of vector bundles over $U_i$. For all $1\leq i, j\leq k$,  the transition map
\be\label{transition}
\psi_j^{-1} \circ \psi_i : (U_i\cap U_j)\times V_0\rightarrow (U_i\cap U_j)\times V_0,
\ee
is given by
\[
 (x,v)\mapsto (x,(s_j(x)^{-1}s_i(x)).v).
\]
 By Lemmas \ref{smodg} and \ref{temcompullb}, \eqref{transition} is a tempered bundle map. Thus the local charts $\{(U_i,V_0,\psi_i)\}_i$ are pairwise compatible and all the local charts of $G\times ^{H}V_0$ that are compatible with all $(U_i,V_0,\psi_i)$'s form a tempered structure on $G\times ^{H}V_0$. It is easy to see that this tempered structure is independent of the choice of the finite family $\{(U_i, s_i)\}_{i=1}^k $. Thus $G\times ^{H}V_0$ is canonically a tempered vector bundle over $G/H$. Moreover, it is easily checked that $G\times ^{H}V_0$ is  in fact a tempered left $G$-vector bundle over $G/H$, under the obvious actions of $G$.

\section{A proof of Theorem \ref{thmA11}}\label{firs}

Recall that $G$ and $\g$ act on $\CS(G)$ by right translations.

\begin{lem}\label{thmA1000}
Assume that $G$ is connected. Then
\[
\dim \oH_{0}(\g; \CS(G))=1.
\]
\end{lem}

\begin{proof}
By Poincar\'{e} duality for Lie algebra cohomologies, we have that
\be\label{pd01}
\oH_{0}(\g; \CS(G)) \cong \oH^{\dim \g}(\g; \CS(G)\otimes\wedge^{\dim \g}\g).
\ee
Note that
\[
  \CS(G)\otimes\wedge^{\dim \g}\g\cong \CS(G)
\]
as representations of $G$, and hence they are isomorphic to each other as $\g$-modules. Thus
\be\label{pd02}
  \oH^{\dim \g}(\g; \CS(G)\otimes\wedge^{\dim \g}\g)\cong  \oH^{\dim \g}(\g; \CS(G)).
\ee

By \cite[Theorem 4.3]{Pr}, one has that
\be\label{pd03}
\oH^{i}_{\DR,\CS}(G)\cong\oH^{i}_{\DR,c}(G), \quad \textrm{ for all $i\in \mathbb Z$},
\ee
where
$\oH^{i}_{\DR,\CS}(G)$ denotes the $i$th de Rham cohomology of the Nash manifold $G$ with Schwartz coefficients and
$\oH^{i}_{\DR,c}(G)$ denotes the $i$th de Rham cohomology of the smooth manifold $G$ with  compactly supported smooth coefficients. Poincar\'{e} duality for de Rham cohomologies implies that
\be\label{pd04}
  \dim \oH^{\dim \g}_{\DR,c}(G)=1.
\ee
By comparing the complexes computing the cohomologies, one has that
\be\label{pd05}
  \oH^{i}_{\DR,\CS}(G)\cong  \oH^{i}(\g; \CS(G)).
\ee
Combining \eqref{pd01}, \eqref{pd02}, \eqref{pd03}, \eqref{pd04} and \eqref{pd05}, the lemma follows.
\end{proof}

Recall that $E$ is  a \Fre space, and $G$ acts on $\CS(G, E)$ by right translations.

\begin{prpl}\label{thmA100}
Assume that $G$ is connected. Then
\[
\g.\CS(G, E)=\left\{f\in\CS(G, E) : \int_G f(g)\rd_r g = 0\right\}.
\]
\end{prpl}

\begin{proof}
Obviously, one has that
\[
\g.\CS(G, E)\subset\left\{f\in\CS(G, E) : \int_G f(g)\rd_r g = 0\right\}.
\]
Thus the integration map
\[
  \CS(G, E)\rightarrow E, \quad f\mapsto \int_G f(g) \rd_r g,
\]
which is surjective, descends to a surjective linear map
\be\label{je}
  J_E: \oH_{0}(\g; \CS(G, E)) \rightarrow E.
\ee
It suffices to show that the above map is a linear isomorphism.

When $E=\mathbb C$, the linear map \eqref{je} becomes a surjective linear map
\be\label{je2}
  J_\mathbb C: \oH_{0}(\g; \CS(G)) \rightarrow \mathbb C.
\ee
Then Lemma \ref{thmA1000} implies that  \eqref{je2} is a linear isomorphism.

In general, we have an obvious commutative diagram
\[
 \begin{CD}
           \oH_{0}(\g; \CS(G))\widehat \otimes E@>J_\mathbb C \otimes 1_E  >> \mathbb C\otimes E \\
            @VVV           @VV =V\\
          \oH_{0}(\g; \CS(G, E)) @>  J_E >> E,\\
  \end{CD}
  \]
where $1_E$ denotes the identity map of $E$. It follows from \cite[II, \S 2, $n^\circ $1]{Gr1} that the left vertical arrow of the above diagram is a linear isomorphism.  Since the top horizontal arrow is also a linear isomorphism, the proposition follows.

\end{proof}

Similar to \eqref{defodg}, let $\oD_c^\infty(G)$ denote the space of compactly supported smooth densities on $G$, equivalently,
\[
 \oD_c^\infty(G)=\con_c^\infty(G)\cdot \rd_r g.
\]
It is a subalgebra of $\oD_c^{-\infty}(G)$.  Dixmier--Malliavin's Theorem \cite[Theorem 3.3]{DM} asserts that
\[
V= \oD_c^\infty(G). V,
\]
for every smooth \Fre representation $V$ of $G$.

\begin{lem}\label{thmA10000}
Assume that $G$ is connected. Then
\[
\left\{f\in\CS(G, E) : \int_G f(g)\rd_r g = 0\right\}=\sum_{g\in G} (g-1). \CS(G, E).
\]
\end{lem}

\begin{proof}
One has that
\begin{eqnarray*}
&& \left\{f\in\CS(G, E):\int_G f(g)\rd_r g = 0\right\}\\
  &=& \g.\CS(G, E)\qquad \quad \textrm{by Proposition \ref{thmA100}}\\
       &=& \g.(\oD_c^\infty(G).\CS(G, E))\qquad \quad \textrm{by Dixmier--Malliavin's Theorem}\\
        &=& (\g\cdot \oD_c^\infty(G)).\CS(G, E))\\
          &=& \sum_{g\in G}((g-1)\cdot \oD^{\infty}_{c}(G)).\CS(G, E)\qquad \quad \textrm{by Theorem \ref{thmA0}} \\
           &=& \sum_{g\in G}(g-1).(\oD^{\infty}_{c}(G).\CS(G, E))\\
            &=& \sum_{g\in G}(g-1).\CS(G, E) \qquad \quad\textrm{by Dixmier--Malliavin's Theorem.}\\
\end{eqnarray*}

\end{proof}

Lemma \ref{thmA10000} holds without the assumption that $G$ is connected, as in the following proposition.

\begin{prpl}\label{thmA100000}
For every almost linear Nash group $G$ and every \Fre space $E$,
\be\label{int0s}
\left\{f\in\CS(G, E) : \int_G f(g)\rd_r g = 0\right\}=\sum_{g\in G} (g-1). \CS(G, E).
\ee
\end{prpl}

\begin{proof}
Let $G^\circ$ denote the identity connected component of $G$. Write
\[
  G=G^\circ\, \sqcup \, G^\circ g_1\, \sqcup\, \cdots \, \sqcup G^\circ g_k, \qquad (k\geq 0),
\]
where $g_1, \cdots, g_k\in G$.
It is easy to see that the space of the left hand side of \eqref{int0s} equals
\[
   \left\{f\in\CS(G^\circ, E)\, :\, \int_G f(g)\rd_r g = 0\right\} \oplus  \bigoplus_{i=1}^k(g_i-1) .\CS(G^\circ, E).
     \]
Here $\CS(G^\circ, E) $ is viewed as a subspace of $\CS(G, E) $, by extension by zero. Thus the proposition follows from Lemma \ref{thmA10000}.
\end{proof}

\begin{lem}\label{ele}
Let $A\subset G$ and denote by $\la A\ra$ the subgroup of $G$ generated by $A$. For every representation $V$ of $G$,
\[
\sum_{a\in A}(a-1).V = \sum_{a\in \la A\ra }(a-1).V .
\]
\end{lem}

\begin{proof}
Every (non-necessarily continuous)  linear functional on $V$ fixed by the subset $A\subset G$ is fixed by the subgroup $\la A\ra $. This implies the lemma.
\end{proof}

\begin{lem}\label{repr}
Assume that $G$ is connected. Then
\[
\g.V = \sum_{g\in G}(g-1).V,
\]
for every representation $V$ in $\CS\mathrm{mod}_{G}$.
\end{lem}

\begin{proof}
By Propositions  \ref{thmA100000} and \ref{thmA100},
\[
\g.\CS(G, E) = \sum_{g\in G}(g-1).\CS(G, E).
\]
The lemma then follows from the observation that every representation $V$ in $\CS\mathrm{mod}_{G}$
can be realized as a quotient of $\CS(G, V)$ (Here $G$ acts on $\CS(G, V)$ by the right translations on the $G$ factor).
\end{proof}

\begin{prpl}\label{ccc}
For every subset $A\subset G$ that intersects all connected components of $G$,
\[
\sum _{g\in G}(g-1).\CS (G, E) = \g.\CS (G, E) + \sum_{a\in A}(a-1).\CS (G, E).
\]
\end{prpl}

\begin{proof}
Recall that $G^{\circ}$ denotes the identity connected component of the group $G$. By Lemmas \ref{repr} and \ref{ele},
\begin{eqnarray*}
 \g.\CS(G, E) + \sum_{a\in A}(a-1).\CS(G, E) &=& \sum_{g\in G^\circ\cup A}(g-1).\CS(G, E)\\
   &=& \sum_{g\in G}(g-1).\CS(G, E).
\end{eqnarray*}
\end{proof}

Theorem  \ref{thmA11} is now proved by combining Propositions  \ref{thmA100000} and \ref{ccc}.
It has the following interesting  consequence.

\begin{corl}
Let $G$ be an almost linear Nash group. Then every right invariant linear functional on $\CS(G)$ is automatically continuous.
\end{corl}

\begin{proof}
Every right invariant linear functional $\phi$ on $\CS(G)$ factors through the coinvariant space  $\CS(G)_{G}$, which is a 1-dimensional Hausdorff topological vector space by Theorem \ref{thmA11}. Thus $\phi$ must be continuous.
\end{proof}

\section{Schwartz homologies}\label{sh}

\subsection{Relatively projective representations}

Recall from Definition \ref{rp} the notion of relatively projective representations in $\CS\mathrm{mod}_G$ and also
the notion of principal left (or right)  $G$-Nash bundles from Example \ref{exm1234}. Let $\mathrm{d}_lg$ denote a fixed left invariant Haar measure on $G$.

\begin{lem}\label{chi}
For every principal  right $G$-Nash bundle $M$ and every \Fre space  $E$, there exists a smooth function
$\chi$ on $M$ with the following properties:
\begin{itemize}
\item the linear map
\[
   \mathcal{S}(M, E)\rightarrow \mathcal{S}(G\times M, E), \quad f\mapsto((g,x)\mapsto \chi(xg)f(x))
\]
is well-defined and continuous; \smallskip
\item $\int_{G}\chi(xg)\mathrm{d}_lg =1$ for all $x\in M$.
\end{itemize}
\end{lem}

\begin{proof}
We prove the trivial bundle case by assuming that $M = Y \times G$, where $Y$ is a Nash manifold and $G$ acts on $M$ by right translations on the second factor.
Note that the general case follows from a local trivialization technique (see, for example the proof of \cite[Proposition 7]{Bl}) .  Take a function $f\in\mathcal{S}(G)$ with $\int_G f(g)\mathrm{d}_l g= 1$. Define $\chi(y,g) := f(g)$. It is easy to check that $\chi$  has the two properties of the lemma.
\end{proof}

\begin{prpl}\label{repro}
Let $M$ be a principal right $G$-Nash bundle and let $E$ be a representation in $\CS\mathrm{mod}_G$. Then $\CS(M, E)$ is a relatively projective representation in $\CS\mathrm{mod}_G$, with the action given  by
\be\label{prin}
(g.f)(x):=g.f(xg),\quad \textrm{ $g\in G$, $x\in M$ and $f\in\CS(M, E)$.}
\ee
\end{prpl}

\begin{proof}
It follows from \cite[Propositions 1.4.4 and 1.4.5]{Fd} that $\CS(M, E)$ is in $\CS\mathrm{mod}_G$. The proof that $\CS(M, E)$ is relatively projective is similar to that of \cite[Theorem 6]{Bl}. We just sketch it here. Let $\chi\in C^{\infty}(M)$ be as in Lemma \ref{chi}. For each
$f\in\CS(M, E)$, define
\[
f_{\chi}\in\CS(G\times M, E)=\CS(G, \CS(M, E))
\]
 by
\[
f_{\chi}(g,x):=\chi(xg)f(x), \quad g\in G, x\in M.
\]

Let  $\alpha : V_1\rightarrow V_2$ be a  surjective strong homomorphism in $\CS\mathrm{mod}_G$, and let $\beta:\CS(M, E)\rightarrow V_2$  be a homomorphism in $\CS\mathrm{mod}_G$.
Take a  continuous linear section $\tau: V_2\rightarrow V_1$ of $\alpha$ (which may or may not be $G$-equivariant).  Define a map
\[
\tilde {\beta}:\CS(M, E)\rightarrow V_1, \quad f\mapsto\int_{G}g.\tau(g^{-1}.\beta(f_{\chi}(g)))\mathrm{d}_lg.
\]
It is easy to check that $\tilde{\beta}$ is a homomorphism in $\CS\mathrm{mod}_G$ which lifts $\beta$. This proves the proposition.
\end{proof}

\begin{remark}
If $M$ is a principal left $G$-Nash bundle and $E$ is a representation in $\CS\mathrm{mod}_G$, then $\CS(M, E)$ is also a relatively projective representation in $\CS\mathrm{mod}_G$, with the action as in \eqref{actgf}.
\end{remark}

For every \Fre space $E$, write $\CS(G, E)_{l}$ for the \Fre space $\CS(G, E)$ carrying the representation of $G$ by the left translations; and write
$\CS(G, E)_{r}$ for the same space carrying the representation of $G$ by the right translations.
More generally, given a representation $E$ in $\CS\mathrm{mod}_G$, there are four natural actions of $G$ on the \Fre space $\CS(G, E)$:
\[
  (g.f)(x):= f(g^{-1}x),\, g. (f(g^{-1}x)),\,  f(xg), \,\textrm{ or } \, g.(f(xg)),
  \]
  for all $g, x\in G$, $f\in \CS(G, E)$. We respectively write $\CS(G, E)_{l1}$, $\CS(G, E)_{l2}$, $\CS(G, E)_{r1}$ and $\CS(G, E)_{r2}$ for the resulting representations of $G$.
When the action of $G$ on $E$ is  trivial, we have that
\[
\CS(G, E)_{l1}=\CS(G, E)_{l2}= \CS(G, E)_{l}
\]
and
\[
\CS(G, E)_{r1} = \CS(G, E)_{r2}=\CS(G, E)_{r},\]
 as representations of $G$.

\begin{lem}\label{spe pro}
Let $E$ be a  representation  in $\CS\mathrm{mod}_G$. Then
\[
  \CS(G, E)_{l1}\cong\CS(G, E)_{l2}\cong \CS(G, E)_{r1}\cong\CS(G, E)_{r2}
\]
as representations of $G$, and they are relatively projective in
$\CS\mathrm{mod}_G$.
\end{lem}

\begin{proof}
It is clear that
\[
  \CS(G, E)_{l1}\rightarrow \CS(G, E)_{r1},\quad f\mapsto (g\mapsto f(g^{-1}))
\]
is an isomorphism of representations of $G$. Likewise,
\[
  \CS(G, E)_{l1}\rightarrow \CS(G, E)_{l2},\quad f\mapsto (g\mapsto g.f(g)),
\]
and
\[
  \CS(G, E)_{r1}\rightarrow \CS(G, E)_{r2},\quad f\mapsto (g\mapsto g^{-1}.f(g)),
\]
are also  isomorphisms of representations of $G$. This proves the first assertion of the lemma. The second assertion follows from Proposition \ref{repro}.
\end{proof}

Lemma  \ref{spe pro} and the following proposition imply that the category $\CS\mathrm{mod}_G$ has enough relatively projective objects.

\begin{prpl}\label{ssurj}
For every representation $E$ in $\CS\mathrm{mod}_G$,
\be\label{ss}
\CS(G, E)_{r2}\rightarrow E, \quad f\mapsto\int_{G}f(g)\mathrm{d}_rg
\ee
 is a  surjective strong homomorphism in the category $\CS\mathrm{mod}_G$.
\end{prpl}

\begin{proof}
Clearly the map \eqref{ss} is a surjective homorphism in $\CS\mathrm{mod}_G$. Pick $\phi\in \CS(G)$ such that
\[
  \int_{G}\phi(g)\mathrm{d}_r g=1.
\]
Then the surjective linear map \eqref{ss} has a  continuous linear section  given by
\[v\mapsto(g\mapsto \phi(g)v).
\]
\end{proof}

\begin{prpl}\label{invrel}
For every relatively projective representation $P$ and every representation $E$ in $\CS\mathrm{mod}_G$, the completed projective tensor product $P\widehat{\otimes}E$ is a relatively projective representation with the diagonal $G$-action.
\end{prpl}

\begin{proof}
By Proposition \ref{ssurj}, every relatively projective representation $P$ is a direct summand of $\CS(G, P)_{r2}$. Since a direct summand of a relatively projective representation is relatively projective,  it is enough to show that
$\CS(G, P)_{r2}\widehat{\otimes}E$ is relatively projective for every representation $E$ in $\CS\mathrm{mod}_G$.
But this is obvious since $$\CS(G, P)_{r2}\widehat{\otimes}E\cong\CS(G, P\widehat{\otimes}E)_{r2}$$ as representations of $G$.
\end{proof}

\begin{prpl}\label{rpr}
A representation in  $\CS\mathrm{mod}_G$ is relatively projective if and only if it is isomorphic to a direct summand of a representation of the form
$\CS(G, E)_r$, where $E$ is a \Fre space. \end{prpl}

\begin{proof}
Since a direct summand of a relatively projective representation is relatively projective, the if part of the proposition follows from
Lemma  \ref{spe pro}.
The only if part follows from Lemma  \ref{spe pro} and Proposition \ref{ssurj}.
\end{proof}

The following result will be useful later.

\begin{prpl}\label{61}
When $G$ is compact, every representation in  $\CS\mathrm{mod}_G$ is relatively projective.
\end{prpl}

\begin{proof}
This is well known. For a proof, see  \cite[Lemma 7]{CW} for example.
\end{proof}

\subsection{The coinvariants of relatively projective representations}\label{co}

Before going to the definition of Schwartz homology of representations in $\CS\mathrm{mod}_{G}$, we prove a remarkable property of relatively projective representations.

\begin{prpl}\label{AAA}
With the notation as in Proposition \ref{repro},
\[
\sum_{g\in G}(g-1).\CS(M, E) = \left\{f\in\CS(M, E)\, : \, \int_{G}g.(f(xg))\mathrm{d}_rg=0,\textrm{ for all $x\in M$}\right\}.
\]
\end{prpl}

\begin{proof}
When $M=G$ and the $G$-action on $E$ is trivial, this is Theorem \ref{thmA11}. For the general case, the proof is analogous to that of \cite[Proposition 7]{Bl}. We sketch the proof of the trivial bundle case  for the convenience of the reader. So we assume that $M=Y\times G$, where $Y$ is a Nash manifold and $G$ acts on $M$ by the right translations on the second factor.

We have the topological linear  isomorphisms
\be\label{reduced}
\CS(M, E)\cong\CS(M, E)\cong\CS(G, \CS(Y, E)).
\ee
Here the first isomorphism is $f\mapsto \tilde f$, where
\[
\tilde{f}(y,g)=g.f(y,g)\quad \textrm{for all }(y,g)\in M=Y\times G,
\]
 The second isomorphism is $f\mapsto \bar f$, where
\[
\bar{f}(g)(y):=f(y,g) \quad \textrm{for all }(y,g)\in Y\times  G.
\]
Let $G$ act on the first $\CS(M, E)$ as in \eqref{prin}, act on the second $\CS(M, E)$ by the right translations on $M$, and act on
$\CS(G, \CS(Y, E))$ by the right translations on $G$.
It is easy to check that the isomorphisms in \eqref{reduced} are $G$-equivariant.
For every $f\in\CS(M, E)$, it is clear that
\[
  \int_{G}g.f(xg)\mathrm{d}_rg=0, \quad \textrm{for all $x\in M$}
  \]
  if and only if
 \[
 \int_{G}\bar{\tilde{f}}(g)\mathrm{d}_rg=0.
 \]
 Thus the proposition (in the trivial bundle case) follows from  Theorem \ref{thmA11}.
\end{proof}

\begin{remark}
Proposition \ref{AAA} will be crucial in our characterization of ``Schwartz induced representations", see Proposition \ref{Sind}.
\end{remark}

\begin{thml}\label{hausrelpro}
For every relatively projective representation $V$ in $\CS\mathrm{mod}_{G}$, the coinvariant space $V_{G}$ is a \Fre space.
\end{thml}

\begin{proof}
By Proposition \ref{rpr}, $V$ is isomorphic to a direct summand of a representation of the form $\CS(G, E)_r$, where $E$ is a \Fre space. Theorem \ref{thmA11}  implies that the coinvariant space
\[
  (\CS(G, E)_r)_G\cong E
\]
is Hausdorff. This implies that $V_G$ is also Hausdorff.
\end{proof}

In the rest of this subsection, we will give a proof of Theorem \ref{thmA2}.
Recall from \eqref{defodg} that $\oD^\varsigma(G) := \CS(G)\rd_r g$ is  the space of Schwartz densities on $G$.
For  every  \Fre space $E$, as in \eqref{indact}, the induced action of $\oD^\varsigma(G)$ on
$\CS(G, E)_r$ is given by
\be\label{cov}
(\mu.f)(x) := \int_{G} f(xg)\rd \mu(g),\quad \mu\in\oD^{\varsigma}(G),\,f\in\CS(G, E)_r,\,x\in G.
\ee
Recall that $ \mathrm I^\varsigma(G):=\{ \mu \in \oD^\varsigma(G)\, :\, \int_G 1\od\! \mu(g)=0\}.$

\begin{prpl}\label{coinva}
For every subset $A\subset G$ that intersects all connected components of $G$,
\[
\mathrm I^\varsigma(G).\CS(G, E)_r = \sum_{g\in G}(g-1).\CS(G, E)_r = \g.\CS(G, E)_r + \sum_{a\in A}(a-1).\CS(G, E)_r.
\]
\end{prpl}

\begin{proof}
The second equality was already proved in Proposition \ref{ccc}.
On the other hand, Theorem \ref{thmA11} implies that
\[
  \mathrm I^\varsigma(G)=\sum_{g\in G} \oD^\varsigma(G)\cdot (g-1).
\]
By applying the push-forward map of measures through the map $G\rightarrow G, x\mapsto x^{-1}$, the above equality implies that
\[
  \mathrm I^\varsigma(G)=\sum_{g\in G} (g-1) \cdot \oD^\varsigma(G).
\]
Thus by
Dixmier--Malliavin's Theorem  \cite[Theorem 3.3]{DM}, one has that
\begin{eqnarray*}
\mathrm I^\varsigma(G).\CS(G, E)_r &=& \sum_{g\in G}((g-1)\cdot \oD^\varsigma(G)).\CS(G, E)_r\\
   &=& \sum_{g\in G}(g-1).\CS(G, E)_r.\\
\end{eqnarray*}
\end{proof}

As in the proof of Lemma \ref{repr}, Lemma  \ref{spe pro} and Proposition \ref{ssurj} imply that  every  representation $V$ in $\CS\mathrm{mod}_{G}$  can be realized as a quotient of $\CS(G, V)_{r1}$.  Then the first assertion of Theorem \ref{thmA2} follows from  Proposition \ref{coinva}.  The second assertion has already been established in Theorem \ref{hausrelpro}.

\subsection{Schwartz homologies}\label{2.4}
In this subsection, we present a Schwartz homology theory for representations in the category $\CS\mathrm{mod}_{G}$.

\begin{dfnl}\label{spr}
For every representation $V$ in the category $\CS\mathrm{mod}_{G}$, a strong projective resolution of $V$ is  an exact sequence
\[
\cdots \rightarrow P_2\rightarrow P_1\rightarrow P_0 \rightarrow V\rightarrow 0
\]
 in $\CS\mathrm{mod}_{G}$, where   $P_i$'s are all relatively projective, and all the arrows are strong homomorphisms.
\end{dfnl}

The existence of strong projective resolutions of every representation  in the category $\CS\mathrm{mod}_{G}$ follows directly from
Proposition \ref{ssurj}.

\begin{dfnl}\label{sm homo}
For every representation $V$ in $\CS\mathrm{mod}_{G}$ and every $i\in \BZ$, the $i$th Schwartz homology  $\oH^{\CS}_{i}(G; V)$ of $V$ is  defined to be the $i$th homology of the chain complex $(P_{\bullet})_{G}$, where $P_{\bullet}\rightarrow V\rightarrow 0$ is a strong projective resolution of $V$.
\end{dfnl}

 We equip the Schwartz homology  $\oH^{\CS}_{i}(G; V)$ with the subquotient topology from the complex $(P_{\bullet})_{G}$. As a locally convex topological vector spaces,  it does not depend on the choice of the strong projective resolution $P_{\bullet}\rightarrow V\rightarrow 0$ of $V$. This follows from the comparison theorem as in \cite[Section 2]{HM} or \cite[Section 4]{Bl}. For a homomorphism
$\alpha:V_{1}\rightarrow V_{2}$ in $\CS\mathrm{mod}_{G}$, the induced continuous linear map $\tilde{\alpha}:\oH^{\CS}_{i}(G; V_{1})\rightarrow\oH^{\CS}_{i}(G; V_{2})$ is canonically defined as done in classical homological algebra, see for
example \cite[Section 2]{HM}.

\begin{prpl}\label{0homo}
For every representation $V$ in $\CS\mathrm{mod}_{G}$, $\oH^{\CS}_{0}(G; V)= V_{G}$ as topological vector spaces.
\end{prpl}

\begin{proof}
Let
\[
\cdot\cdot\cdot\rightarrow P_{1}\rightarrow P_{0}\overset{\alpha}{\rightarrow} V\rightarrow 0
\]
be a strong projective resolution of $V$ in $\CS\mathrm{mod}_{G}$. Since taking coinvariants is right exact, we have the exact sequence
\[
(P_{1})_{G}\rightarrow (P_{0})_{G}\overset{\tilde{\alpha}}{\rightarrow} V_{G}\rightarrow 0.
\]
By the open mapping theorem \cite[Theorem 17.1]{Tr}, the map $\alpha$ is open, then so is $\tilde{\alpha}$. Now it follows easily that $\tilde{\alpha}$ induces a topological linear isomorphism from $\oH^{\CS}_{0}(G; V)$ to $V_{G}$.
\end{proof}


\section{Schwartz sections and Schwartz inductions}\label{sectvb00}

\subsection{Schwartz sections}\label{secsch}

Recall that tempered vector bundles have been defined in Section \ref{sectvb}.
In this subsection, we define Schwartz sections of a tempered vector bundle. This generalizes  the definition of Schwartz sections of Nash vector bundles, see \cite[Section 5]{AG1}.

Let $(M, \mathsf E, \CT_{\mathsf E})$ be a tempered vector bundle. Suppose that  $\{(U_i, E_i, \phi_i)\}_{i=1}^k$ ($k\geq 0$) are elements of
$\CT_{\mathsf E}$ such that $\{U_i\}_{i=1}^{k}$ is a covering of $M$. Write $\Gamma^{\varsigma}(U_{i},U_{i}\times E_{i})$ for the space of the sections which correspond to Schwartz functions in
$\CS(U_i, E_i)$. This is obviously a \Fre space.
Define
\[
  \Gamma^{\varsigma}(U_{i},\mathsf E|_{U_i}):=\{\phi_{i}\circ s\, :\,  s\in\Gamma^{\varsigma}(U_{i},U_{i}\times E_{i})\},
  \]
  which is also obviously  a \Fre space.

   Denote by
$\Gamma(M, \mathsf E)$ the space of continuous sections of the bundle $\mathsf E$ over $M$. Then extension by zero gives a continuous linear map
\be\label{schsect}
\bigoplus_{i=1}^k\Gamma^{\varsigma}(U_{i},\mathsf E|_{U_i})\rightarrow\Gamma(M,\mathsf E).
\ee

\begin{dfn}
With the notation as above, the Schwartz sections $\Gamma^{\varsigma}(M, \mathsf E)$ of the tempered vector bundle $\mathsf E$ over the Nash manifold $M$ is defined to be the image of the map \eqref{schsect}, equipped with the quotient topology of the domain.
\end{dfn}

\begin{prpd}
The definition of  $\Gamma^{\varsigma}(M, \mathsf E)$ (as a topological vector space) does not depend on the choice of the local charts $\{(U_i, E_i, \phi_i)\}_{i=1}^k$.
\end{prpd}

\begin{proof}
The proof is similar to that of \cite[Proposition 5.1.2]{AG1}.
\end{proof}

\begin{prpd}\label{sss}
Suppose that the Nash manifold $M$ carries a left Nash $G$-action. Let $\mathsf E$ be a tempered left $G$-vector bundle over $M$. Then for every $g\in G$ and $\phi\in \Gamma^{\varsigma}(M, \mathsf E)$,
\be\label{actgphi}
    \begin{array}{rcl}
g.\phi:  M&\rightarrow &\mathsf E,\\
 x&\mapsto & g.(\phi(g^{-1}.x))
 \end{array}
\ee
is a section in $\Gamma^{\varsigma}(M, \mathsf E)$. Moreover,  the \Fre space $\Gamma^{\varsigma}(M, \mathsf E)$  is a representation in
$\CS\mathrm{mod}_G$ under the action
\[
  (g, \phi)\mapsto g.\phi.
\]
\end{prpd}

\begin{proof}
We have an obvious tempered vector bundle $G\times \mathsf E$ over $G\times M$. Moreover,
\[
    \Gamma^{\varsigma}(G\times M, G\times \mathsf E)=\CS(G)\widehat \otimes  \Gamma^{\varsigma}(M, \mathsf E).
\]
Write $ \Gamma_1^{\varsigma}(G\times M, G\times \mathsf E):= \Gamma^{\varsigma}(G\times M, G\times \mathsf E)$, to be viewed as a representation of $G$ under the  left translations of $G$ on $\CS(G)$. This is a  representation in $\CS\mathrm{mod}_{G}$.

Note that the map
\[
  G\times \mathsf E\rightarrow G\times \mathsf E, \quad (g, v)\mapsto (g, g.v)
\]
is an isomorphism of tempered vector bundles over the Nash isomorphism
\[
  G\times M\rightarrow G\times M, \quad (g,x)\mapsto (g, g.x).
\]
Thus it induces a topological linear automorphism
\[
    \begin{array}{rcl}
   \eta:   \Gamma^{\varsigma}(G\times M, G\times \mathsf E)&\rightarrow & \Gamma^{\varsigma}(G\times M, G\times \mathsf E),\\
       \phi&\mapsto &((g,x)\mapsto (g, g. (\phi(g, g^{-1}.x)))).\\
       \end{array}
\]
Write $ \Gamma_2^{\varsigma}(G\times M, G\times \mathsf E):= \Gamma^{\varsigma}(G\times M, G\times \mathsf E)$, to be viewed as a representation of $G$ such that
\[
   \eta:   \Gamma_1^{\varsigma}(G\times M, G\times \mathsf E)\rightarrow  \Gamma_2^{\varsigma}(G\times M, G\times \mathsf E)
   \]
   is an isomorphism of representations of $G$. Then
   $\Gamma_2^{\varsigma}(G\times M, G\times \mathsf E)$ is also a representation in $\CS\mathrm{mod}_{G}$.

Define an action of $G$ on $\Gamma(M,\mathsf E)$ as in \eqref{actgphi}.
Now define a linear map
\[
    \begin{array}{rcl}
 \mathrm J:    \Gamma_2^{\varsigma}(G\times M, G\times \mathsf E)&\rightarrow &\Gamma(M, \mathsf E),\\
       \phi&\mapsto &\left(x\mapsto \int_G \phi(g, x) \od\!_l g\right),\\
       \end{array}
\]
where $ \od\!_l g$ is a fixed left invariant Haar measure on $G$. It is clear that the map $\mathrm J$ is $G$-equivariant, and its image equals $\Gamma^{\varsigma}(M, \mathsf E)$. Thus $\Gamma^{\varsigma}(M, \mathsf E)$ is $G$-stable in $\Gamma(M, \mathsf E)$.  This proves the first assertion of the proposition.

Finally, $\mathrm J$ induces a  $G$-equivariant  linear map
 \[
 \mathrm J:    \Gamma_2^{\varsigma}(G\times M, G\times \mathsf E)\rightarrow \Gamma^{\varsigma}(M, \mathsf E).
 \]
 This map is surjective and continuous, and hence open by the open mapping Theorem. Therefore,  $\Gamma^{\varsigma}(M, \mathsf E)$ is a quotient representation of $\Gamma_2^{\varsigma}(G\times M, G\times \mathsf E)$, and it is a representation in  $\CS\mathrm{mod}_G$ by Lemma \ref{sq}.
\end{proof}

Let $\mathsf E$ be a tempered vector bundle over the Nash manifold $M$, then for every  open Nash submanifold  $U$ of $M$, $ \mathsf E|_U$ is  a tempered vector bundle over $U$. The extension by zero yields a continuous linear map
\be\label{ezero}
\mathrm{Ex}_U^M :  \Gamma^{\varsigma}(U, \mathsf E|_U)\rightarrow \Gamma^{\varsigma}(M, \mathsf E).
\ee
The image of the  map \eqref{ezero} may be characterized as in \cite[Theorem 5.4.1]{AG1}. Roughly speaking,  the image consists of all the sections which vanish with all its derivatives outside $U$. In particular, the map \eqref{ezero} is a closed embedding.

\begin{prpd}
Let $\{U_i\}_{i=1}^k$ ($k\geq 0$) be a finite cover of $M$ by its open Nash submanifolds. Then the sequence
\[
  \bigoplus_{1\leq i, j\leq k} \Gamma^{\varsigma}(U_i\cap U_j, \mathsf E|_{U_i\cap U_j})\rightarrow  \bigoplus_{1\leq i\leq k} \Gamma^{\varsigma}(U_i, \mathsf E|_{U_i})\rightarrow  \Gamma^{\varsigma}(M, \mathsf E)\rightarrow 0
\]
is  exact. Here the first arrow is the linear map specified by requiring that
\[
  \phi\mapsto \mathrm{Ex}_{U_i\cap U_j}^{U_i}(\phi)-\mathrm{Ex}_{U_i\cap U_j}^{U_j}(\phi)
\]
for every $\phi\in \Gamma^{\varsigma}(U_i\cap U_j, \mathsf E|_{U_i\cap U_j})$.

\end{prpd}

\begin{proof}
The proof is similar to that of \cite[Proposition 5.1.3]{AG1}.
\end{proof}

\subsection{Schwartz inductions}\label{schind}

In this subsection, we recall the notion of Schwartz inductions in the sense of du Cloux, see \cite[Section 2]{Fd}. Then we show that
they are the same to the Schwartz produced representations \eqref{produced} as defined in the Introduction, up to twists by characters.

Let $H$ be a Nash subgroup of $G$ and let $V_0$ be a smooth representation of $H$. As in Example \ref{exm12},  viewing $V_0$ as a representation of $G$ with the trivial action, $\con^{\infty}(G, V_0)$ is a smooth representation of $G$ under the left translations. Write
\[
\Ind_{H}^{G}V_0:=\{\varphi\in \con^{\infty}(G, V_0)\, |\, \varphi(gh)=h^{-1}.\varphi(g),\textrm{ for all $h\in H$, $g\in G$}\}
\]
for the  unnormalized smooth induction. It is a subrepresentation of $\con^{\infty}(G, V_0)$.

Suppose that $V_0$ is in $\CS\mathrm{mod}_{H}$.   Define a $G$-equivariant continuous linear map
\be\label{schinduced}
\begin{array}{rcl}
 \Phi :  \CS(G, V_0)_l&\rightarrow& \Ind_{H}^{G}V_0,\\
\varphi&\mapsto& \left(g\mapsto \int_{H}h.\varphi(gh)\mathrm{d}_{l}h\right),
\end{array}
\ee
where $\mathrm{d}_{l}h$ is a   left invariant Haar measure  on $H$.
Denote by $\ind_{H}^{G}V_0$ the image of the map \eqref{schinduced},  equipped with the quotient topology of
the domain. Then $\ind_{H}^{G}V_0$ is a representation in $\CS\mathrm{mod}_{G}$, and we call it the Schwartz induced representation of $V_0$.


\begin{remarks} (See \cite[Remark 2.1.4] {Fd}.) $ $\\
(a) If $G/H$ is compact, then $\ind_{H}^{G}V_0=\Ind_{H}^{G}V_0$.\\
(b)  $\ind_{H}^{G}\BC=\CS(G/H)$ as representations of $G$.
\end{remarks}

Every homomorphism
$\rho: V_{1}\rightarrow V_{2}$ in the category
$\CS\mathrm{mod}_{H}$ induces a homomorphism $$\ind_H^G{\rho}:\ind_{H}^{G}V_{1}\rightarrow\ind_{H}^{G}V_{2}$$ in
$\CS\mathrm{mod}_{G}$. It is clear that $\ind_H^G$ is a functor from the category $\CS\mathrm{mod}_{H}$ to the category $\CS\mathrm{mod}_{G}$.

As in \eqref{delta}, let $\delta_H$ denote the modular character of $H$. Using Proposition \ref{AAA}, we have another  characterization of Schwartz inductions as in the following proposition.

\begin{prpd}\label{Sind}
For every representation $V_0$ in $\CS\mathrm{mod}_{H}$, there is an isomorphism
\[
 ( \CS(G, V_0)_l\otimes\delta_{H}^{-1})_{H}\cong\ind_{H}^{G}V_0
\]
of representations of $G$. Here $H$ acts on $\CS(G, V_0)$ as in \eqref{actgf02}.
\end{prpd}

\begin{proof}
Consider the composition of
\[
   \CS(G, V_0)_l\otimes\delta_{H}^{-1}\xrightarrow{\varphi\otimes 1\mapsto \varphi}   \CS(G, V_0)_l\xrightarrow{\textrm{the map \eqref{schinduced}}}  \Ind_{H}^{G}V_0.
\]
It  follows from Proposition \ref{AAA} that the kernel of this composition map equals
$$\sum_{h\in H}(h-1).(\CS(G, V_0)\otimes\delta_{H}^{-1}). $$
Hence the proposition holds.
\end{proof}

Recall from the Introduction  the Schwartz produced representation $$\mathrm{pro}_{H}^{G}V_0 = (\oD^\varsigma(G)\widehat{\otimes}V_0)_H.$$
The following result follows directly from Proposition \ref{Sind}.

\begin{prpd}\label{pro=ind}
For every representation $V_0$ in $\CS\mathrm{mod}_{H}$,
\[
\mathrm{pro}_{H}^{G}V_0\cong (\ind_{H}^{G}(V_0\otimes\delta_{H}))\otimes\delta_{G}^{-1}
\]
as representations of $G$.
\end{prpd}

Recall from Section \ref{htvb} that $G\times^{H}V_0$ is a tempered $G$-vector bundle
over $G/H$, and by  Proposition \ref{sss}, $\Gamma^{\varsigma}(G/H,G\times^{H}V_0)$ is a representation in  $\CS\mathrm{mod}_{G}$.

\begin{prpd}\label{geometry}
For every representation $V_0$ in $\CS\mathrm{mod}_{H}$,
$$\Gamma^{\varsigma}(G/H,G\times^{H}V_0)\cong\ind_{H}^{G}V_0$$ as representations of $G$.
\end{prpd}

\begin{proof}
As in Section \ref{htvb}, choose a finite family $\{(U_i, s_i)\}_{i=1}^k $ ($k\geq 0$) such that $\{U_i\}_{i=1}^k $ is a covering of $G/H$ by its open Nash submanifolds, and $s_i$ is a Nash section
of the quotient map $G\rightarrow G/H$ over $U_i$. Write $G_i$ for the preimage of $U_i$ under the quotient map $G\rightarrow G/H$. Then
\[
  G_i=s_i(U_i)\times H,
\]
and
\[
  \CS(G, V_0)=\sum_{i=1}^k \CS(G_i, V_0).
\]

As usual, identify $\Ind_H^G V_0$ with the space $\Gamma^{\infty}(G/H,G\times ^{H}V_0)$ of the smooth sections of the bundle $G\times ^{H}V_0$. It is clear that the image of  $\CS(G_i, V_0)$ under the map
\eqref{schinduced} equals
\[
  \Gamma^{\varsigma}(U_i,(G\times ^{H}V_0)|_{U_i})\subset \Gamma^{\varsigma}(G/H,G\times ^{H}V_0)\subset \Gamma^{\infty}(G/H,G\times ^{H}V_0).
\]
Thus the image of the map \eqref{schinduced} equals $\Gamma^{\varsigma}(G/H,G\times ^{H}V_0)$, since
\[
  \Gamma^{\varsigma}(G/H,G\times ^{H}V_0)=\sum_{i=1}^k  \Gamma^{\varsigma}(U_i,(G\times ^{H}V_0)|_{U_i}).
\]
In view of the open mapping theorem, this proves the proposition.
\end{proof}

\subsection{Frobenious reciprocity}\label{frobe}

Let $H$ be a Nash subgroup of $G$ as before. Now we prove the following version of Frobenious reciprocity, which is Theorem \ref{fro} of the Introduction.

\begin{thmd}\label{fro222}
 Let $V_0$ be a representation in $\CS\mathrm{mod}_H$. Then the continuous linear map
\be\label{110}
  \oD^\varsigma(G)\widehat \otimes V_0\rightarrow V_0, \quad \mu\otimes v\mapsto \int_G 1 \od\! \mu(g) \cdot v
\ee
induces an identification
\[
  (\mathrm{pro}_H^G V_0)_G=(V_0)_H
\]
of topological vector spaces.
\end{thmd}

\begin{proof}
Theorem \ref{thmA11}  implies that the map \eqref{110} descends to an identification
\[
 (\oD^\varsigma(G)\widehat \otimes V_0)_G=V_0,
\]
where $G$ acts on $\oD^\varsigma(G)\widehat \otimes V_0$ by the left translations on $\oD^\varsigma(G)$.
Thus we have identifications
\begin{eqnarray*}
  &&(\mathrm{pro}_H^G V_0)_G\\
  &= & ((\oD^\varsigma(G)\widehat \otimes V_0)_H)_G\\
  &= &  ((\oD^\varsigma(G)\widehat \otimes V_0)_G)_H\\
    &=&  (V_0)_H.\\
\end{eqnarray*}
\end{proof}

\begin{remark}
Denote
$\delta_{G/H}:=(\delta_{G})|_H\cdot\delta_{H}^{-1}$. Theorem \ref{fro222} and Proposition \ref{pro=ind} imply that
\[
(\ind_{H}^{G}V_0)_G\cong(V_0\otimes\delta_{G/H})_{H}.
\]
\end{remark}

\section{More on Schwartz homologies}\label{secsha}

In this section, we go back to the Schwartz homologies of representations of an almost linear Nash group as defined in Section \ref{sh}. The main result here is Shapiro's lemma, an important tool for computing the Schwartz homologies of Schwartz produced (induced) representations. We will also prove Theorem \ref{hosho} and discuss the Schwartz homologies of finite dimensional representations.

\subsection{Schwartz induction and relatively projectiveness}
Recall that $G$ is an almost linear Nash group.
Let $H$ be a Nash subgroup of $G$, the Schwartz induction functor $\ind_{H}^{G}:\CS\mathrm{mod}_{H}\rightarrow\CS\mathrm{mod}_{G}$ was defined in Section \ref{schind}.

The following proposition was proved in \cite[Proposition 2.2.7]{Fd}.

\begin{prp}\label{exact}
The Schwartz induction functor $\ind_{H}^{G}:\CS\mathrm{mod}_{H}\rightarrow\CS\mathrm{mod}_{G}$ is  exact, and maps strong homomorphisms to strong homomorphisms.
\end{prp}

The following proposition was proved in \cite[Lemma 2.1.6]{Fd}.

\begin{prp}\label{exact22}
For every Nash subgroup $H'$ of $H$, and every representation $V_0'$ in $\CS\mathrm{mod}_{H'}$, there  is a natural  isomorphism
\[
  \ind_{H}^{G}(\ind_{H^\prime}^{H}V_0')\cong\ind_{H^\prime}^{G}V_0'
\]
of representations of $G$.
\end{prp}

The following proposition says that the Schwartz induction functor preserves  relatively projective representations.

\begin{prp}\label{rproj}
For every relatively projective representation $P$ in $\CS\mathrm{mod}_{H}$, $\ind_{H}^{G}P$ is relatively projective in
$\CS\mathrm{mod}_{G}$.
\end{prp}

\begin{proof}
This follows from Propositions \ref{rpr} and \ref{exact22}.
\end{proof}

\begin{remark}
The analogous result as in Proposition \ref{rproj}  holds in the $p$-adic case, see \cite[Lemma 4.3]{HS} for example.
\end{remark}

We shall also need the following result of Schwartz inductions.

\begin{prp}\label{ccls}
Let $V_{0}$ and $V$ be smooth moderate growth \Fre representations of $H$ and $G$, respectively. If $V$ or $V_0$ is nuclear, then there is an isomorphism
\[
\ind_{H}^{G}(V_{0}\widehat{\otimes}V|_H)\cong(\ind_{H}^{G}V_{0})\widehat{\otimes}V
\]
 of representations of $G$.
\end{prp}

\begin{proof}
If $V$ is nuclear, this is \cite[Lemma 3.2]{LS}. If $V_{0}$ is nuclear, this follows from the arguments of the proof of \cite[Lemma 3.2]{LS}, together with the following fact:
Let $E_1\rightarrow E_2$ be an injective continuous linear map of nuclear  \Fre spaces, and let $E_3$ be a \Fre space, then the induced map $E_1\widehat \otimes E_3\rightarrow E_2 \widehat \otimes E_3$ is also injective. This fact is implied by \cite[Equation (50.17)]{Tr} and the fact that nuclear \Fre spaces are reflexive.
\end{proof}

\subsection{Shapiro's lemma}\label{shapi}

In this subsection, we will prove Theorem \ref{fro2}. We restate it here for the convenience of the reader.

\begin{thmp}[Shapiro's Lemma]\label{shapiro}
Let $H$ be a Nash subgroup of an almost linear Nash group $G$, and let $V_0$ be a representation in $\CS\mathrm{mod}_H$.  Then there is an identification
\[
\oH_{i}^\CS(G;  \mathrm{pro}_H^G V_0)=\oH_{i}^\CS(H; V_0)
\]
of topological vector spaces, for all $i\in \mathbb Z$.
\end{thmp}

\begin{proof}
Let $P_{\bullet}\rightarrow V_0\rightarrow 0$ be a strong projective resolution of $V_0$ in the category $\CS\mathrm{mod}_{H}$. By
Propositions \ref{pro=ind}, \ref{exact}  and \ref{rproj} we conclude that
\[
\mathrm{pro}_{H}^{G}P_\bullet\rightarrow\mathrm{pro}_{H}^{G}V_0\rightarrow 0
\]
is a strong projective resolution of the representation $\mathrm{pro}_{H}^{G}V_0$ in the category $\CS\mathrm{mod}_{G}$. Now it follows from the Frobenious reciprocity ( Theorem \ref{fro}) that
\[
(\mathrm{pro}_{H}^{G}P_{\bullet})_G = (P_{\bullet})_{H}
\]
as chain complexes, and thus the theorem follows.
\end{proof}

\begin{remark}
The smooth homology theory of smooth representations of real Lie groups was established in \cite{Bl} by P. Blanc. One of the main results there is a form of Shapiro's lemma, concerning the compactly supported smooth induced representation, see \cite[Theorem 11]{Bl}.
\end{remark}

\subsection{A resolution of the trivial representation}

Fix a maximal compact subgroup $K$ of $G$. As before, denote by $\g$ and $\k$ the complexified Lie algebras of $G$ and $K$, respectively.

Let $X:=G/K$, and write $\mathrm T^* X$ for its cotangent bundle. Put
\[
  \Omega_{X}^{k}:=\wedge^k (\mathrm T^* X\otimes_\BR \BC), \quad (k\in \BZ).
\]
Denote $n:=\dim X$. Write $Or_{X}$ for the orientation line bundle of $X$, with complex coefficients. Its fibre at $1K\in X$ equals
\[
   \wedge^n(\g/\k)\otimes \abs{\wedge^n(\g/\k)^*}, \qquad ( \textrm{ a superscript $*$ indicates  the dual space}).
\]
Here and as usual, for every one dimensional complex vector space $F$, $\abs{F}$ denotes a one dimensional complex vector space equipped with a nonzero map $\abs{\,\cdot \,}: F\rightarrow \abs{F}$ such that
\[
  \abs{a.v}=\abs{a}. \abs{v}\quad \textrm{for all }a\in \BC, v\in F.
\]

 Both  $\mathrm T^* X$ and $ Or_{X}$ are  obviously  tempered left $G$-vector bundles over $X$, and hence so is
$\Omega_{X}^{k}\otimes  Or_{X}$. We have the extended de Rham complex
\be\label{deR}
0\rightarrow\Gamma^\varsigma(X, \Omega_{X}^{0}\otimes Or_{X})\rightarrow\cdot\cdot\cdot\rightarrow \Gamma^\varsigma(X, \Omega_{X}^{n}\otimes Or_{X}){\xrightarrow{\textrm{integration} }}\BC\rightarrow 0.
\ee
Note that all the arrows in \eqref{deR} are homomorphisms of representations of $G$, and \be\label{geo}
 \Gamma^\varsigma(X, \Omega_{X}^{k}\otimes Or_{X})=\ind_{K}^{G}(\wedge^k(\g/\k)^*\otimes \wedge^n(\g/\k)\otimes \abs{\wedge^n(\g/\k)^*}).
\ee
By Propositions \ref{61} and \ref{rproj}, the representations in \eqref{geo} are relatively projective in $\CS\mathrm{mod}_G$.
The following proposition says that the sequence \eqref{deR} gives a strong projective resolution of the trivial representation $\BC$ in
$\CS\mathrm{mod}_G$.

\begin{prp}\label{stracy}
The sequence \eqref{deR} is exact, and all the homomorphisms in the sequence are strong homomorphisms in $\CS\mathrm{mod}_G$.
\end{prp}

\begin{proof}
Consider the following de Rham complex with compactly supported smooth coefficients:
\be\label{deR1}
0\rightarrow\Gamma^{\infty}_{c}(X, \Omega_{X}^{0}\otimes Or_{X})\rightarrow\cdot\cdot\cdot\rightarrow \Gamma^{\infty}_{c}(X, \Omega_{X}^{n}\otimes Or_{X}){\xrightarrow{\textrm{integration} }}\BC\rightarrow 0.
\ee
Since $X$ (as a smooth manifold) is diffeomorphic to $\BR^n$, de Rham had constructed an explicit contracting homotopy for the complex \eqref{deR1}, see \cite[Section 5]{dR}. Note that $X$ is actually Nash diffeomorphic to $\BR^n$, and de Rham's  construction also applies to the de Rham complex \eqref{deR} with Schwartz coefficients.
Hence the sequence \eqref{deR} is exact, and all the homomorphisms in the sequence are strong homomorphisms.
\end{proof}

\subsection{Schwartz homologies and $(\g, K)$-homologies}\label{iden}
In this subsection, we will show that for representations in $\CS\mathrm{mod}_G$,  the Schwartz homologies  as defined in Section \ref{2.4} coincide with the relative Lie algebra homologies.

The following Theorem is Theorem \ref{hosho} of the Introduction.

\begin{thmp}\label{BBB}
For every representation $V$ in the category $\CS\mathrm{mod}_{G}$, there is an identification
\[
\oH^{\CS}_{i}(G; V) = \oH_{i}(\g, K; V)
\]
of topological vector spaces, for all $i\in \mathbb Z$.
\end{thmp}

\begin{proof}
The argument is similar to that of \cite[Theorem 6.1]{HM}.  We sketch the proof for the convenience of the reader.
By \cite[Theorem 5.24]{Ta}, \eqref{deR} induces an exact sequence
\be\label{deR2}
0\rightarrow\Gamma^\varsigma(X, \Omega_{X}^{0}\otimes Or_{X})\widehat \otimes V\rightarrow\cdot\cdot\cdot\rightarrow \Gamma^\varsigma(X, \Omega_{X}^{n}\otimes Or_{X})\widehat \otimes V{\xrightarrow{\textrm{integration} }}V\rightarrow 0
\ee
in $\CS\mathrm{mod}_G$. By Proposition \ref{stracy}, we conclude that every homomorphism in the sequence \eqref{deR2} is strong. By
Proposition \ref{invrel}, we know that each representation $\Gamma^\varsigma(X, \Omega_{X}^{k}\otimes Or_{X})\widehat \otimes V$ is relatively projective in $\CS\mathrm{mod}_G$. Thus the sequence \eqref{deR2} is a strong projective resolution of the representation $V$ in
$\CS\mathrm{mod}_G$. Now according to \eqref{geo}, Proposition \ref{ccls} and Theorem \ref{fro222}, we have that
\be\label{gind}
(\Gamma^\varsigma(X, \Omega_{X}^{k}\otimes Or_{X})\widehat \otimes V)_G\cong(\wedge^{n-k}(\g/\k)\otimes V)_{K}.
\ee
With the isomorphisms in \eqref{gind}, one verifies that the chain complex $$\{\Gamma^\varsigma(X, \Omega_{X}^{n-i}\otimes Or_{X})\widehat \otimes V\}_{i\in \BZ}$$ coincides with the chain complex computing the relative Lie algebra homology of $V$. This proves the theorem.
\end{proof}

\begin{remark}
This kind of result is known as van Est theorem, see \cite[Theorem 2]{vE}. Theorem \ref{BBB} will be useful in showing vanishing of the Schwartz homologies. We will give an application of this result in the next section.
\end{remark}

\begin{corp}\label{sil}
Every short exact sequence $0\rightarrow V_{1}\rightarrow V_{2}\rightarrow V_{3}\rightarrow 0$ in the category
$\CS\mathrm{mod}_{G}$ yields a long exact sequence
\[
\cdots \rightarrow\oH^{\CS}_{i+1}(G; V_{3})\rightarrow\oH^{\CS}_{i}(G; V_{1})\rightarrow\oH^{\CS}_{i}(G; V_{2})\rightarrow \oH^{\CS}_{i}(G; V_{3})\rightarrow\cdots
\]
of (non-necessary Hausdorff)  locally convex topological vector spaces.
\end{corp}

\begin{proof}
This follows from Theorem \ref{BBB} and the corresponding result for relative Lie algebra homologies.
\end{proof}

\subsection{Schwartz homology of finite dimensional representations}\label{secfinite}

In this subsection, we will discuss finite dimensional representations. Firstly, we recall some structure theory of almost linear Nash groups.

Recall that a finite dimensional real or complex representation of a Nash group is  said to be a Nash representation if the action map is a Nash map.
A Nash group is called reductive if it has a completely reducible Nash representation with finite kernel. It is known that all Nash representations of reductive Nash groups are completely reducible.  A Nash group is called unipotent if it has a faithful Nash representation such  that all the group elements act as unipotent linear operators.  A maximal reductive Nash subgroup of the  almost linear Nash group $G$ is called a Levi component of $G$. It is unique up to conjugation. The unipotent radical $N$ of $G$ is defined to be the largest normal unipotent Nash subgroup of $G$. Then  we have the Levi decomposition $G=L\ltimes N$, where $L$ denotes a Levi component of $G$. For these facts, see \cite[Theorems 1.16 and 1.17]{Su}.

\begin{lemp}\label{aut moder}
Every finite dimensional representation of a reductive Nash group is of moderate growth.
\end{lemp}

\begin{proof}
This follows from a more general result: every continuous Banach representation of a reductive group is of moderate growth. See for example \cite[Section 2.2]{Wa}.
\end{proof}

\begin{prp}
A finite dimensional representation $F$ of $G$ is of moderate growth if and only if every irreducible subquotient of $F|_N$ is unitarizable, where $N$ denotes the unipotent radical of $G$.
\end{prp}

\begin{proof}
Firstly, from the Levi decomposition $G=LN$ and Lemma \ref{aut moder}, it is enough to prove the proposition when  $G$ is unipotent. Thus we assume that $G$ is unipotent.

Recall from  \cite[Corollary 6.7]{Fd1} that the moderate growth property is preserved by extensions of representations  and by taking subquotients. Thus, without loss of generality,  we further assume that $F$ is irreducible.  Now the proposition follows from \cite[Theorem 5.1]{Fd1}.
\end{proof}

\begin{remark}
As a Lie group, every unipotent Nash group  is connected, simply connected and nilpotent (see \cite[Theorem 1.8]{Su}). Thus we can apply the results of du Cloux \cite{Fd1}.
By the above proposition, a character of a unipotent Nash group  is of moderate growth if and only if it is unitary.
\end{remark}

Note that $\oH^{\CS}_p(L^\circ; \BC)$ and $\oH^{\CS}_q(N; \BC)$ are naturally Nash representations of $L$, where $p,q\in \BZ$, $L^\circ$ is the identity connected component of $L$, and $\BC$ stands for the trivial representations. Therefore, $\oH^{\CS}_p(L^\circ; \BC)\otimes\oH^{\CS}_q(N; \BC)$ is also a Nash representation of $L$, and consequently it is completely reducible.

\begin{prp}\label{vani}
Let $F$ be an irreducible finite dimensional representation of $G$ of moderate growth, and let $i\in \BZ$. Assume that either the action of $N$ on $F$ is not trivial, or $N$ acts trivially on $F$ and  the irreducible representation $F^\vee|_L$ (the restriction to $L$ of  contragredient representation of $F$) does not occur  in
\[
\bigoplus_{p,q\in \BZ, p+q=i} \oH^{\CS}_p(L^\circ; \BC)\otimes \oH^{\CS}_q(N; \BC).
\]
Then $\oH^{\CS}_i(G; F)=0$.
\end{prp}

\begin{proof}
First assume that  the action of $N$ on $F$ is not trivial. Then  the irreducibility condition and  Lie's theorem on solvable Lie algebras imply that the representation  $F|_N$ is a direct sum of nontrivial one-dimensional representations.  By a  spectral sequence argument, this  implies  that 
\[
\textrm{$\oH^{\CS}_q(N; F)=0\ $ for all $q\in \BZ$, $\quad$ and $\quad$ $\oH^{\CS}_i(G; F)=0$.}
\]

Now we assume that $N$ acts trivially on $F$. Then for all $p,q\in \BZ$, 
\begin{eqnarray*}
  &&\oH^{\CS}_p(L; \oH^{\CS}_q(N; F))\\
  &\cong &\oH^{\CS}_p(L; \oH^{\CS}_q(N; \BC)\otimes F)\\
  &\cong &\left(\oH^{\CS}_p(L^\circ; \oH^{\CS}_q(N; \BC)\otimes F)\right)_L\quad (\textrm{the coinvariant space})\\
  &\cong &\left(\oH^{\CS}_p(L^\circ; \BC)\otimes \left(\oH^{\CS}_q(N; \BC)\otimes F\right)_{L^\circ}\right)_L\quad\textrm{(by \cite[Chapter 1, Theorem 5.3]{BoW})}\\
    &\cong &\left(\oH^{\CS}_p(L^\circ; \BC)\otimes \oH^{\CS}_q(N; \BC)\otimes F\right)_L.
  \end{eqnarray*}
This implies the proposition by a  spectral sequence argument.
\end{proof}

\section{Automatic extensions}

In this section, we will prove the automatic extensions of Schwartz homologies, namely, Theorems \ref{fro3} and \ref{fro4}.
The main tools are Shapiro's lemma (Theorem \ref{shapiro}), and Borel's lemma of the following subsection.


\subsection{Borel's lemma}

We begin with the following definition.

\begin{dfnl}\label{kvanishing}
Let $M$ a smooth manifold,  and let $E$ a quasi-complete Hausdorff locally convex topological vector space over $\BC$. A $E$-valued smooth function $f$ on $M$ is said to be $k$-vanishing ($k\geq 1$) at a point $x\in M$ if for every differential operator $D$ on $M$ of order $\leq k-1$, $(Df)(x)=0$.
\end{dfnl}

Now suppose that  $M$ is  a Nash manifold and $\mathsf E$ is a tempered vector bundle over $M$. For every $\phi\in \Gamma^{\varsigma}(M, \mathsf E)$ and every $x\in M$,  the notion that $\phi$ is  $k$-vanishing at $x$ is obviously defined by using Definition \ref{kvanishing} and a local chart $(U, E, \psi)$ of $\mathsf E$ with $x\in U$. Moreover, this notion is independent of the choice of the local chart.

For each $x\in M$, define
\[
  \Gamma^{\varsigma}(M, \mathsf E)_{x,k}:=\{\phi\in \Gamma^{\varsigma}(M, \mathsf E)\, :\, \phi \textrm{ is $k$-vanishing at $x$}\}.
\]
This is a closed subspace of $ \Gamma^{\varsigma}(M, \mathsf E)$. For convenience, write
\[
   \Gamma^{\varsigma}(M, \mathsf E)_{x,0}:= \Gamma^{\varsigma}(M, \mathsf E).
\]

Now suppose that $U$ is an open Nash submanifold of $M$.  Write $Z:=M\setminus U$. As in \eqref{ezero}, extension by zero yields a closed  linear embedding
\be\label{extby}
\Gamma^{\varsigma}(U, \mathsf E|_U)\hookrightarrow\Gamma^{\varsigma}(M, \mathsf E),
\ee
and we identify
$\Gamma^{\varsigma}(U, \mathsf E|_U)$ with its image in $\Gamma^{\varsigma}(M, \mathsf E)$. Define
\[
\Gamma_{Z}^{\varsigma}(M, \mathsf E):=\Gamma^{\varsigma}(M, \mathsf E)/\Gamma^{\varsigma}(U, \mathsf E|_U).
\]
 For every $k\geq 0$, put
\[
  \Gamma_{Z}^{\varsigma}(M, \mathsf E)_k:= \left(\bigcap_{x\in Z} \Gamma^{\varsigma}(M, \mathsf E)_{x,k}\right)/\Gamma^{\varsigma}(U, \mathsf E|_U)\subset \Gamma^{\varsigma}_Z(M, \mathsf E).
  \]
This is  a closed subspace of $ \Gamma_Z^{\varsigma}(M, \mathsf E)$.

\begin{prpl}\label{cha}
The natural map  $$\Gamma^{\varsigma}_Z(M, \mathsf E)\rightarrow \varprojlim_k \, \Gamma_Z^{\varsigma}(M, \mathsf E)/\Gamma_{Z}^{\varsigma}(M, \mathsf E)_k$$
is a topological linear  isomorphism.

\end{prpl}
\begin{proof}
This is a form of Borel's lemma. See \cite[Lemma A.2.8]{AG2} for a proof when $\mathsf E$ is a Nash bundle and $Z$ is a closed Nash submanifold. When $Z$ is a closed Nash submanifold, the same proof works in our general setting of tempered vector bundles. The general case is easily reduced to this case by considering a filtration
\[
  Z=Z_0\supset Z_1\supset \cdots \supset Z_r\supset Z_{r+1}=\emptyset, \quad (r\geq 0)
\]
such that for all $0\leq i\leq r$, $Z_i$ is a closed semialgebraic subset of $M$, and $Z_i\setminus Z_{i+1}$ is a Nash submanifold of $M$.

\end{proof}

For each $k\geq 0$ and $x\in M$, in what follows we define a bilinear map
 \be\label{dx}
 \mathcal D_x:   (\otimes^k  \mathrm T_x(M) )\times \Gamma^{\varsigma}(M, \mathsf E)_{x,k}\rightarrow \mathsf E_x,
 \ee
where $\mathsf E_x$ is the fibre of $\mathsf E$ at $x$, and $\mathrm T_x(M)$ is the tangent space of $M$ at $x$. Let $v_1, v_2, \ldots, v_k\in \mathrm T_x(M)$, and $\phi\in \Gamma^{\varsigma}(M, \mathsf E)_{x,k}$. Take a local chart $(U, \mathsf E_x, \psi)$ of $\mathsf E$ such that $x\in U$ and $$\mathsf E_x\xrightarrow{v\mapsto (x, v)}U\times \mathsf E_x\xrightarrow{\psi} \mathsf E$$ induces an identity map of $\mathsf E_x$. By using this local chart,
we identify $\phi|_U$ with a smooth function $\phi_U: U\rightarrow \mathsf E_x$.
For each $1\leq j\leq k$, take a vector field $Y_j$ on $U$ which extends $v_j$. Now we define
\[
  \mathcal D_x(v_1\otimes v_2\otimes \cdots \otimes v_k, \phi)= ((Y_1 Y_2 \cdot \ldots \cdot Y_k) \phi_U)(x).
\]
This is independent of the local chart $(U, \mathsf E_x, \psi)$ and the vector fields $Y_j$'s.

\begin{remark}
Obviously, the map \eqref{dx} may be defined in a more general setting of smooth manifolds, smooth vector bundles, and smooth sections.
\end{remark}

If  $Z$ is a closed Nash submanifold of $M$, write
\[
  \mathrm{N}_{Z}(M):=\bigsqcup_{x\in Z} \frac{ \mathrm T_x(M)}{\mathrm T_x(Z)}\otimes_\BR \BC
\]
for the complexified  normal bundle of $Z$ in $M$. Write  $\mathrm{N}^*_{Z}(M)$ for its dual bundle, which is called the complexified conormal bundle.

\begin{prpl}\label{cha1}
Suppose that  $Z$ is a closed Nash submanifold of $M$. Then the maps \eqref{dx} for all  $x\in Z$ induces a topological linear isomorphism
\[
  \Gamma_{Z}^{\varsigma}(M, \mathsf E)_k/\Gamma_{Z}^{\varsigma}(M, \mathsf E)_{k+1}\cong \Gamma^\varsigma(Z, \Sym^k(\mathrm N^*_Z(M))\otimes \mathsf E|_Z),
\]
for all integers $k\geq 0$.
\end{prpl}
\begin{proof}
This is also a part of Borel's lemma. See \cite[Lemmas A.2.7]{AG2} for a proof when $\mathsf E$ is a Nash bundle. The same proof works in our general setting of tempered vector bundles.
\end{proof}

Let $G$ be an almost linear Nash group as before. Now suppose that  $M$ is a left $G$-Nash manifold,
$\mathsf E$ is a tempered left $G$-vector bundle over $M$, and  $U$ is a $G$-stable open Nash submanifold of $M$.
Then  $\Gamma^{\varsigma}(M, \mathsf E)$ is naturally a representation in $\CS\mathrm{mod}_G$, and   $\Gamma^{\varsigma}(U, \mathsf E|_U)$ is a subrepresentation of it.
Recall that $Z:=M\setminus U$. For each $k\geq 0$,
$
  \Gamma^{\varsigma}_Z(M, \mathsf E)_k
$
is also a subrepresentation of $\Gamma^{\varsigma}_Z(M, \mathsf E)$.

As in the Introduction, let $\chi: G\rightarrow \mathbb C^\times$ be a character which has moderate growth.
The following lemma is similar to
 \cite[Corollary 2.3.3]{AGKL}.

\begin{lem}\label{commu}
Suppose that  $Z$ is a closed Nash submanifold of $M$. Let $i\in \BZ$ and assume that $\oH_{i+1}^\CS(G;(\Gamma_{Z}^{\varsigma}(M, \mathsf E)/\Gamma_{Z}^{\varsigma}(M, \mathsf E)_k)\otimes\chi)$ is finite dimensional for all integers $k\geq 0$. Then the canonical map
\[
\oH_{i}^\CS(G;\Gamma_{Z}^{\varsigma}(M, \mathsf E)\otimes\chi)\rightarrow\varprojlim_{k}\oH_{i}^\CS(G;(\Gamma_{Z}^{\varsigma}(M, \mathsf E)/\Gamma_{Z}^{\varsigma}(M, \mathsf E)_k)\otimes\chi)
\]
is a  linear isomorphism.
\end{lem}

\begin{proof}
Recall that a sequence
\[
\cdots\rightarrow V_{j+1}\rightarrow V_{j}\rightarrow\cdots \rightarrow V_1\rightarrow V_0
\]
of complex vector spaces  is called a Mittag-Leffler sequence if for each $j_0\geq 0$, the image of the composition of
\[
     V_{j}\rightarrow V_{j-1}\rightarrow\cdots \rightarrow V_{j_0+1}\rightarrow V_{j_0}
\]
is independent of $j$ whenever $j$ is sufficiently large.

For simplicity, write
\[
\Gamma_k:=(\Gamma_{Z}^{\varsigma}(M, \mathsf E)/\Gamma_{Z}^{\varsigma}(M, \mathsf E)_k)\otimes \chi, \quad k\geq 0.
\]
Consider the sequence
\[
\cdots\rightarrow (\wedge^{\bullet}(\g/\k)\otimes\Gamma_2)_K\rightarrow(\wedge^{\bullet}(\g/\k)\otimes \Gamma_1)_K
\rightarrow(\wedge^{\bullet}(\g/\k)\otimes\Gamma_0)_K
\]
of chain complexes.
Obviously, at each degree, the corresponding sequence is a Mittag-Leffler sequence (since the corresponding linear maps are all surjective). The induced sequence of the homologies at the $(i+1)$th degree is also Mittag-Leffler by the finite dimension assumption. Thus by \cite[Chapter 0, Proposition 13.2.3]{Gr} and Theorem  \ref{BBB}, we have that
\begin{eqnarray*}
\oH_{i}(G;\Gamma_{Z}^{\varsigma}(M, \mathsf E)\otimes\chi)&=&\oH_{i}(\g,K;\Gamma_{Z}^{\varsigma}(M, \mathsf E)\otimes\chi)\\
&=&\oH_{i}((\wedge^{\bullet}(\g/\k)\otimes\varprojlim_{k}\Gamma_{k})_K)\\
  &= &\varprojlim_{k}\oH_{i}((\wedge^{\bullet}(\g/\k)\otimes\Gamma_{k})_K)\\
  &=&\varprojlim_{k}\oH_{i}(\g,K;\Gamma_k)\\
    &=&\varprojlim_{k}\oH_{i}(G;\Gamma_k).
\end{eqnarray*}

\end{proof}


\subsection{A proof of Theorem \ref{fro3} }\label{proof113}

We continue with the notation of the last subsection. Recall from the Introduction the complexified normal space
\[
 \mathrm{N}_{z}:=\frac{\mathrm T_z (M)}{\mathrm T_z(G.z)}\otimes_{\mathbb R} \mathbb C
 \]
 and its dual space $\mathrm{N}^*_{z}$.

 \begin{thml}\label{thm333}
Suppose that $Z$ has only finitely many $G$-orbits. Let $i\in \mathbb Z$.  Assume that for all $z\in Z$ and all integers $k\geq 0$,
\be\label{vi}
\left\{
  \begin{array}{l}
    \oH_{i+1}^\CS(G_{z}; \mathsf E_{z}\otimes\Sym^{k}( \mathrm{N}^*_{z})\otimes\delta_{G/G_{z}}\otimes\chi)\ \textrm{ is finite dimensional, and} \medskip \\
    \oH_i^\CS(G_{z}; \mathsf E_{z}\otimes\Sym^{k}( \mathrm{N}^*_{z})\otimes\delta_{G/G_{z}}\otimes\chi)=0,
  \end{array}
\right.
\ee
 where $\Sym^{k}$ indicates the $k$th symmetric power.
Then
\[ 
\oH^\CS_{i}(G; \Gamma_Z^{\varsigma}(M, \mathsf E)\otimes\chi)=0.
\]
\end{thml}

\begin{proof}
First we assume that  $Z$ is a single $G$-orbit.  In view of Proposition \ref{pro=ind}, Proposition  \ref{geometry}, Shapiro's Lemma (Theorem \ref{shapiro}) and Proposition \ref{cha1}, \eqref{vi} implies that for all integers $k\geq 0$,
\[
  \left\{
  \begin{array}{l}
    \oH^\CS_{i+1}(G; (\Gamma_Z^{\varsigma}(M, \mathsf E)_k/\Gamma_Z^{\varsigma}(M, \mathsf E)_{k+1})\otimes\chi)\ \textrm{ is finite dimensional, and} \medskip \\
     \oH^\CS_{i}(G; (\Gamma_Z^{\varsigma}(M, \mathsf E)_k/\Gamma_Z^{\varsigma}(M, \mathsf E)_{k+1})\otimes\chi)=0.
  \end{array}
\right.
\]
Using the long exact sequences as in  Corollary \ref{sil}, we conclude that
for all integers $k\geq 0$,
\[
  \left\{
  \begin{array}{l}
    \oH^\CS_{i+1}(G; (\Gamma_Z^{\varsigma}(M, \mathsf E)/\Gamma_Z^{\varsigma}(M, \mathsf E)_{k+1})\otimes\chi)\ \textrm{ is finite dimensional, and} \medskip \\
     \oH^\CS_{i}(G; (\Gamma_Z^{\varsigma}(M, \mathsf E)/\Gamma_Z^{\varsigma}(M, \mathsf E)_{k+1})\otimes\chi)=0.
  \end{array}
\right.
\]
Then Lemma \ref{commu} implies that
\[
  \oH^\CS_{i}(G; \Gamma_Z^{\varsigma}(M, \mathsf E)\otimes\chi)=0.
\]

In general, take a filtration
\[
  Z=Z_0\supset Z_1\supset \cdots \supset Z_r\supset Z_{r+1}=\emptyset, \quad (r\geq 0)
\]
such that for all $0\leq j\leq r$, $Z_j$ is a closed semialgebraic subset of $M$, and $Z_j\setminus Z_{j+1}$ is a $G$-orbit.
Such a filtration always exists, see for example \cite[Proposition 3.6]{Su}. We have an obvious exact sequence
\[
 0\rightarrow \Gamma_{Z\setminus Z_r}^{\varsigma}(M\setminus Z_r, \mathsf E|_{M\setminus Z_r})\rightarrow \Gamma_Z^{\varsigma}(M, \mathsf E)\rightarrow \Gamma_{Z_r}^{\varsigma}(M, \mathsf E)\rightarrow 0.
\]
By induction on the number of $G$-orbits in $Z$, the theorem follows by using the induced long exact sequence.
\end{proof}

To prove Theorem \ref{fro3}, we also need the following result.

\begin{lem}\label{open}
Let $\phi : \CC_{\bullet}\rightarrow\CC^{\prime}_{\bullet}$ be a morphism of chain complexes of \Fre spaces. Let $i\in \BZ$. If the induced continuous linear map
\be\label{phi}
\phi: \oH_{i}(\CC_{\bullet})\rightarrow\oH_{i}(\CC^{\prime}_{\bullet})
\ee
 is surjective, then it must be an open map.
\end{lem}

\begin{proof}
Write $\CC_\bullet=\{(\CC_i, \partial_i)\}_{i\in \BZ}$. Put
\[
  \mathrm{Im}_{i}:=\partial_{i+1}(\CC_{i+1})\quad \textrm{and}\quad  \mathrm{Ker}_{i}:=\textrm{the kernel of $\partial_i$}.
\]
Then
\[
\mathrm{Im}_{i}\subset\mathrm{Ker}_{i}\subset \CC_i.
\]
Here $\mathrm{Im}_{i}$ is equipped with the quotient topology of $\CC_{i+1}$ and $\mathrm{Ker}_{i}$ is equipped with the subspace
topology of $\CC_i$. Then both are \Fre spaces.
Similarly we have spaces
$$\mathrm{Im}^{\prime}_{i}\subset\mathrm{Ker}^{\prime}_{i}\subset \CC^{\prime}_i.
$$

The commutative diagram
\be\label{dia1}
\xymatrix{
\mathrm{Ker}_{i} \ar[r]^{\phi} \ar[d]^{\subset }         &{\Ker}^{\prime}_{i} \\
\mathrm{Ker}_{i}\oplus\mathrm{Im}^\prime_i  \ar[ur]_{(u,v)\mapsto \phi(u)+v}
}
\ee
descends to a commutative  diagram
\be\label{dia2}
\xymatrix{
\frac{\mathrm{Ker}_{i}}{\mathrm{Im}_{i}}       \ar[r] \ar[d]^{\cong}         &\frac{\mathrm{Ker}^{\prime}_{i}}{\mathrm{Im}^{\prime}_{i}} \\
\frac{\mathrm{Ker}_{i}\oplus\mathrm{Im}^\prime_i}{\mathrm{Im}_{i}\oplus\mathrm{Im}^{\prime}_{i}}  \ar[ur] \, .
}
\ee
The surjectivity of \eqref{phi} implies the surjectivity of the
diagonal arrow of \eqref{dia1}. Thus by the open mapping theorem \cite[Theorem 17.1]{Tr}, this diagonal arrow must be an open map. Therefore, the
diagonal arrow of \eqref{dia2} is also an open map. Since the vertical arrow of \eqref{dia2} is a topological linear isomorphism, the  horizontal arrow of \eqref{dia2} is also an open map.
This proves the lemma.
\end{proof}

\begin{remark}
Lemma \ref{open} does not follow directly from the Banach Open Mapping Theorem since the homologies are not necessarily Hausdorff.
\end{remark}

In view of Theorem \ref{thm333} and Lemma \ref{open}, Theorem \ref{fro3} follows by considering the long exact sequence of the Schwartz homologies (Corollary \ref{sil}) induced by the  obvious exact sequence
\[
 0\rightarrow \Gamma^{\varsigma}(U, \mathsf E|_{U})\otimes \chi \rightarrow \Gamma^{\varsigma}(M, \mathsf E)\otimes \chi \rightarrow \Gamma_{Z}^{\varsigma}(M, \mathsf E)\otimes \chi\rightarrow 0.
\]

Theorem \ref{fro4} follows directly from Theorem \ref{fro3} and Proposition \ref{vani}.











\end{document}